\documentclass[]{etna}

\hypersetup{%
    pdftitle={Rectangular GLT Sequences},
    pdfauthor={Giovanni Barbarino, Carlo Garoni, Mariarosa Mazza, Stefano Serra-Capizzano},
    pdfkeywords={}
    }

\usepackage{graphicx,amsfonts,amssymb,amsmath,enumitem,epstopdf,mathtools,mathrsfs} 

\usepackage{tocloft}
\cftsetindents{section}{0pt}{10pt}
\cftsetindents{subsection}{13pt}{23pt}
\cftsetindents{subsubsection}{36pt}{33pt}

\newcommand{\hh}{{\boldsymbol h}}
\newcommand{\ii}{{\boldsymbol i}}
\newcommand{\jj}{{\boldsymbol j}}
\newcommand{\kk}{{\boldsymbol k}}

\newcommand{\nn}{{\boldsymbol n}}

\newcommand{\btheta}{\boldsymbol\theta}

\newcommand{\bu}{\mathbf 1}
\newcommand{\ee}{{\mathbf e}}
\newcommand{\xx}{{\mathbf x}}

\renewcommand{\SS}{{\mathbb S}}
\newcommand{\UU}{{\mathbb U}}
\newcommand{\VV}{{\mathbb V}}
\newcommand{\uu}{{\mathbf u}}
\newcommand{\vv}{{\mathbf v}}

\newtheorem{notation}[theorem]{Notation}

\title{Rectangular GLT sequences\thanks{%
Giovanni Barbarino is supported in part by an Academy of Finland grant (Suomen Akatemian p\"a\"at\"os 331240). Carlo Garoni acknowledges the MIUR Excellence Department Project awarded to the Department of Mathematics of the University of Rome Tor Vergata (CUP E83C18000100006), and the support obtained by the Beyond Borders Programme of the University of Rome Tor Vergata through the Project ASTRID (CUP E84I19002250005). All authors are members of the Research Group GNCS (Gruppo Nazionale per il Calcolo Scientifico) of INdAM (Istituto Nazionale di Alta Matematica).}}

\author{Giovanni Barbarino\footnotemark[2]
        \and Carlo Garoni\footnotemark[3]
        \and Mariarosa Mazza\footnotemark[4]
        \and Stefano~Serra-Capizzano\footnotemark[5]}

\shorttitle{Rectangular GLT sequences}
\shortauthor{G. Barbarino, C. Garoni, M. Mazza, and S. Serra-Capizzano}

\begin{document}

\maketitle

\renewcommand{\thefootnote}{\fnsymbol{footnote}}

\footnotetext[2]{Department of Mathematics and Systems Analysis, Aalto University, Espoo, Finland\newline ({\tt giovanni.barbarino@aalto.fi}).}
\footnotetext[3]{Department of Mathematics, University of Rome Tor Vergata, Rome, Italy\newline ({\tt garoni@mat.uniroma2.it}).}
\footnotetext[4]{Department of Humanities and Innovation, University of Insubria, Como, Italy\newline ({\tt mariarosa.mazza@uninsubria.it}).}
\footnotetext[5]{Department of Humanities and Innovation, University of Insubria, Como, Italy; Department of Information Technology, Uppsala University, Uppsala, Sweden\newline ({\tt stefano.serrac@uninsubria.it}; {\tt stefano.serra@it.uu.se}).}

\begin{abstract}
The theory of generalized locally Toeplitz (GLT) sequences is a powerful apparatus for computing the asymptotic spectral distribution of square matrices $A_n$ arising form the discretization of differential problems. Indeed, as the mesh fineness parameter $n$ increases to $\infty$, the sequence $\{A_n\}_n$ often turns out to be a GLT sequence.
In this paper, motivated by recent applications, we further enhance the GLT apparatus by developing a full theory of rectangular GLT sequences as an extension of the theory of classical square GLT sequences. We also detail an example of application as an illustration of the potential impact of the theory presented herein.
\end{abstract}

\begin{keywords}
asymptotic distribution of singular values and eigenvalues, rectangular Toeplitz matrices, rectangular generalized locally Toeplitz matrices, discretization of differential equations, finite elements, tensor products, B-splines
\end{keywords}

\begin{AMS}
15A18, 15B05, 47B06, 65N30, 15A69, 65D07
\end{AMS}

\vspace{-19pt}
\tableofcontents
\vspace{8pt}


\section{Introduction}\label{introduzione}

Suppose that a linear differential problem is discretized using a mesh characterized by a fineness parameter $n$. In this case, the computation of the numerical solution reduces to solving a linear discrete problem---e.g., a linear system or an eigenvalue problem---identified by a square matrix $A_n$. The size of $A_n$ grows as $n$ increases, i.e., as the mesh is progressively refined, and ultimately we are left with a sequence of matrices $A_n$ such that $\textup{size}(A_n)\to\infty$ as $n\to\infty$.
What is often observed in practice is the following.
\begin{itemize}[leftmargin=*]
	\item As long as the considered mesh enjoys a certain structure, 
	the sequence $\{A_n\}_n$ is structured as well and, in particular, it falls in the class of generalized locally Toeplitz (GLT) sequences \cite{barbarinoREDUCED,bg,bgd,GLT-bookI,GLT-bookII}. Depending on the considered problem, $\{A_n\}_n$ could be a traditional scalar GLT sequence \cite{GLT-bookI}, a multilevel GLT sequence \cite{GLT-bookII}, a block GLT sequence \cite{bg}, a multilevel block GLT sequence \cite{bgd}, or a reduced (multilevel block) GLT sequence \cite{barbarinoREDUCED}.
	\item The eigenvalues of $A_n$ enjoy an asymptotic distribution described by a function $f$, in the sense of Definition~\ref{def-distribution}.
	The function $f$, known as the spectral symbol of $\{A_n\}_n$, normally coincides with the so-called kernel (or symbol) of the GLT sequence $\{A_n\}_n$ and can be computed precisely through the theory of GLT sequences.
\end{itemize}
The theory of GLT sequences is therefore an apparatus---to the best of the authors' knowledge, the most powerful apparatus---for computing the spectral symbol $f$ of sequences of matrices $\{A_n\}_n$ arising form the discretization of differential problems. The spectral symbol in turn is useful for several purposes, ranging from the design of appropriate solvers for the considered discretization matrices to the analysis of the spectral approximation properties of the considered discretization method; see \cite[Section~1.2]{bg} and \cite[Section~1.1]{GLT-bookI} for more details. 

Nowadays, the main references for the theory of GLT sequences and the related applications are the books \cite{GLT-bookI,GLT-bookII} and the review papers \cite{barbarinoREDUCED,bg,bgd}. We therefore refer the reader to these works for a comprehensive treatment of the topic, whereas for a more concise introduction to the subject, we recommend the papers \cite{SbMath,axioms,cime,aip}. From a theoretical point of view, among the main recent developments not included in \cite{barbarinoREDUCED,bg,bgd,GLT-bookI,GLT-bookII}, we mention the equivalence between GLT sequences and measurable functions \cite{GiovanniLAA2017}, the normal form of GLT sequences \cite{GiovanniNF}, the perturbation results for GLT sequences \cite{barbarinoNH}, the analysis of the connections between the spectral symbol and the spectrum of the operator associated with the considered differential problem \cite{bianchi,bianchi0,Tom-paper}, and the first ``bridge'' between spectral symbols and spectral analysis of graphs/networks \cite{MJM}. From an applicative point of view, among the main recent developments not included in \cite{barbarinoREDUCED,bg,bgd,GLT-bookI,GLT-bookII}, we mention the application to fractional differential equations \cite{DMS,DMS2} and incompressible Navier--Stokes equations \cite{DFF,NM-iNS}.

Despite the remarkable development that the theory of GLT sequences has reached nowadays, recent applications \cite{ashkanCMAME,NM-iNS} suggested the need for a notion of rectangular GLT sequences in order to further enhance the GLT apparatus. In this paper, we introduce such a notion and develop a full theory of rectangular (multilevel block) GLT sequences as an extension of the theory of classical square (multilevel block) GLT sequences.
We also detail an example of application as an illustration of the potential impact of the theory presented herein.

To give an a priori flavor of the relevance of the theory of rectangular GLT sequences, consider the applications that inspired this paper, i.e., the Taylor--Hood stable finite element (FE) discretization of the linear elasticity equations \cite{ashkanCMAME} and the staggered discontinuous Galerkin approximation of the incompressible Navier--Stokes equations \cite{NM-iNS}. In these cases, the numerical solution is computed by solving a linear system whose coefficient matrix has a saddle point structure of the form
\[ A_n = \begin{bmatrix}A_n(1,1) & A_n(1,2)\\ A_n(2,1) & A_n(2,2)\end{bmatrix}. \]
An efficient solution of this system relies on block Gaussian elimination and essentially reduces to solving a linear system whose coefficient matrix is the Schur complement
\[ S_n = A_n(2,2)-A_n(2,1)(A_n(1,1))^{-1}A_n(1,2); \]
see \cite[Section~5]{BGL}. What is relevant to us is that the sequences $\{A_n(i,j)\}_n$ are, up to minor transformations, square GLT sequences for $i=j$ and rectangular GLT sequences for $i\ne j$. As a consequence, the spectral distributions of $\{A_n\}_n$ and $\{S_n\}_n$ can be computed through the theory of rectangular GLT sequences, and especially properties \underline{\textbf{GLT\,4}} and \underline{\textbf{GLT\,6}} in Section~\ref{summ}, which allow us to ``connect'' GLT sequences with symbols of different size. In \cite{ashkanCMAME,NM-iNS}, the authors computed the spectral distributions of $\{A_n\}_n$ and $\{S_n\}_n$ by either resorting to the complicated technique of ``cutting matrices'' employed in the convergence analysis of multigrid methods or using specific results that are special cases of the theory developed herein. These approaches were adopted as workarounds to the lack of a theory of rectangular GLT sequences, they are somehow application dependent, and ultimately they are intrinsically ``wrong''. The ``right'' approach---more natural, more general, and simpler---is the one we will present in Section~\ref{appl}, which fully exploits the theory of rectangular GLT sequences.

The paper is organized as follows.
In Section~\ref{back}, we collect some background material along with preliminary notations and results.
In Section~\ref{eo}, we introduce and study the extension operator, i.e., the key tool for transferring results about square GLT sequences to rectangular GLT sequences.
In Section~\ref{sec:GLT}, we develop the theory of rectangular GLT sequences, which is then summarized in Section~\ref{summ}.
We conclude in Section~\ref{appl} with an illustrative example of application of the presented theory. 

\section{Preliminaries}\label{back}


\subsection{General notation and terminology}\label{noterm}


\begin{itemize}[leftmargin=*]

\item If $\alpha_1,\ldots,\alpha_n\in\mathbb R$, we define $\alpha_1\wedge\cdots\wedge\alpha_n=\min(\alpha_1,\ldots,\alpha_n)$ and $\alpha_1\vee\cdots\vee\alpha_n=\max(\alpha_1,\ldots,\alpha_n)$.


\item We denote by $\ee_1^{(n)},\ldots,\ee_n^{(n)}$ the vectors of the canonical basis of $\mathbb C^n$.

\item $O_{m,n}$, $O_n$ and $I_n$ denote, respectively, the $m\times n$ zero matrix, the $n\times n$ zero matrix and the $n\times n$ identity matrix. Sometimes, when the sizes can be inferred from the context, $O$ is used instead of $O_{m,n}$, $O_n$, and $I$ is used instead of $I_n$.

\item For every $r,s\in\mathbb N=\{1,2,\ldots\}$ and every $\alpha=1,\ldots,r$ and $\beta=1,\ldots,s$, we denote by $E_{\alpha\beta}^{(r,s)}$ the $r\times s$ matrix having 1 in position $(\alpha,\beta)$ and 0 elsewhere, and we set $E_{\alpha\beta}^{(s)}=E_{\alpha\beta}^{(s,s)}$. 

\item The eigenvalues of a matrix $X\in\mathbb C^{n\times n}$ are denoted by $\lambda_i(X)$, $i=1,\ldots,n$. The singular values of a matrix $X\in\mathbb C^{m\times n}$ are denoted by $\sigma_i(X)$, $i=1,\ldots,m\wedge n$. 
The maximum and minimum singular values of $X$ are also denoted by $\sigma_{\max}(X)$ and $\sigma_{\min}(X)$. 

\item For every $X\in\mathbb C^{m\times n}$, we denote by $\|X\|=\sigma_{\max}(X)$ the spectral (Euclidean) norm of $X$, by $X^*$ the conjugate transpose of $X$, and by $X^\dag$ the Moore--Penrose pseudoinverse of $X$.

\item $C_c(\mathbb C)$ (resp., $C_c(\mathbb R)$) is the space of complex-valued continuous functions defined on $\mathbb C$ (resp., $\mathbb R$) with bounded support.

\item $\mu_k$ denotes the Lebesgue measure in $\mathbb R^k$. Throughout this work, unless otherwise stated, all the terminology from measure theory (such as ``measurable set'', ``measurable function'', ``a.e.'', etc.)\ is always referred to the Lebesgue measure.

\item Let $D\subseteq\mathbb R^k$. 
An $r\times s$ matrix-valued function $f:D\to\mathbb C^{r\times s}$ is said to be measurable (resp., continuous, a.e.\ continuous, bounded, in $L^p(D)$, in $C^\infty(D)$, etc.)\ if its components $f_{\alpha\beta}:D\to\mathbb C$, $\alpha=1,\ldots,r$, $\beta=1,\ldots,s$, are measurable (resp., continuous, a.e.\ continuous, bounded, in $L^p(D)$, in $C^\infty(D)$, etc.). 

\item Let $f_m,f:D\subseteq\mathbb R^k\to\mathbb C^{r\times s}$ be measurable. We say that $f_m$ converges to $f$ in measure (resp., a.e., in $L^p(D)$, etc.)\ if $(f_m)_{\alpha\beta}$ converges to $f_{\alpha\beta}$ in measure (resp., a.e., in $L^p(D)$, etc.)\ for all $\alpha=1,\ldots,r$ and $\beta=1,\ldots,s$. 


\item We use a notation borrowed from probability theory to indicate sets. For example, if $f,g:D\subseteq\mathbb R^k\to\mathbb C^{r\times s}$, then  
$\{f$ has full rank$\}=\{\mathbf x\in D:f(\xx)$ has full rank$\}$, $\mu_k\{\|f-g\|\ge\varepsilon\}$ is the measure of the set $\{\mathbf{x}\in D:\|f(\mathbf{x})-g(\mathbf{x})\|\ge\varepsilon\}$, etc.



\end{itemize}

\subsection{Multi-index notation}\label{ssec:multi-index}

A multi-index $\ii$ of size $d$, also called a $d$-index, is simply a vector in $\mathbb Z^d$. 
\begin{itemize}[leftmargin=*]
	\item $\mathbf0,\bu,\mathbf2,\ldots$ are the vectors of all zeros, all ones, all twos, $\ldots$ (their size will be clear from the context).
	\item For any vector $\nn\in\mathbb R^d$, we set $N(\nn)=\prod_{i=1}^dn_i$ and we write $\nn\to\infty$ to indicate that $\min(\nn)\to\infty$. 
	\item If $\hh,\kk\in\mathbb R^d$, an inequality such as $\hh\le\kk$ means that $h_i\le k_i$ for all $i=1,\ldots,d$.
	\item If $\hh,\kk$ are $d$-indices such that $\hh\le\kk$, the $d$-index range $\{\hh,\ldots,\kk\}$ is the set $\{\ii\in\mathbb Z^d:\hh\le\ii\le\kk\}$. We assume for this set the standard lexicographic ordering:
\begin{equation*}
\Bigl[\ \ldots\ \bigl[\ [\ (i_1,\ldots,i_d)\ ]_{i_d=h_d,\ldots,k_d}\ \bigr]_{i_{d-1}=h_{d-1},\ldots,k_{d-1}}\ \ldots\ \Bigr]_{i_1=h_1,\ldots,k_1}.
\end{equation*}
For instance, in the case $d=2$ the ordering is
\begin{align*}
&(h_1,h_2),\,(h_1,h_2+1),\,\ldots,\,(h_1,k_2),\\
&(h_1+1,h_2),\,(h_1+1,h_2+1),\,\ldots,\,(h_1+1,k_2),\\
&\ldots\,\ldots\,\ldots,\,(k_1,h_2),\,(k_1,h_2+1),\,\ldots,\,(k_1,k_2).
\end{align*}
	\item When a $d$-index $\ii$ varies in a $d$-index range $\{\hh,\ldots,\kk\}$ (this is often written as $\ii=\hh,\ldots,\kk$), it is understood that $\ii$ varies from $\hh$ to $\kk$ following the lexicographic ordering. 
	\item If $\hh,\kk$ are $d$-indices with $\hh\le\kk$, the notation $\sum_{\ii=\hh}^\kk$ indicates the summation over all $\ii=\hh,\ldots,\kk$. 
	\item Operations involving $d$-indices (or general vectors with $d$ components) that have no meaning in the vector space $\mathbb R^d$ must always be interpreted in the componentwise sense. For instance, $\ii\jj=(i_1j_1,\ldots,i_dj_d)$, $\ii/\boldsymbol j=(i_1/j_1,\ldots,i_d/j_d)$, etc.

\end{itemize}

\subsection{Multilevel block matrices}
If $\nn\in\mathbb N^d$ and $X=[x_{\ii\jj}]_{\ii,\jj=\bu}^{\nn}$, where each $x_{\ii\jj}$ is a matrix of size $r\times s$, then $X$ is a matrix of size $N(\nn)r\times N(\nn)s$ whose ``entries'' $x_{\ii\jj}$ are $r\times s$ blocks indexed by a pair of $d$-indices $\ii,\jj$, both varying from $\bu$ to $\nn$ according to the lexicographic ordering. 
Following Tyrtyshnikov \cite[Section~6]{Ty96}, a matrix of this kind is referred to as a $d$-level $(r,s)$-block matrix (with level orders $\nn=(n_1,\ldots,n_d)$).

For every $\nn\in\mathbb N^d$ and every $\ii,\jj=\bu,\ldots,\nn$, we denote by $E_{\ii\jj}^{(\nn)}$ the $N(\nn)\times N(\nn)$ matrix having 1 in position $(\ii,\jj)$ and 0 elsewhere. If $X=[x_{\ii\jj}]_{\ii,\jj=\bu}^{\nn}$ is a $d$-level $(r,s)$-block matrix, then 
\begin{equation}\label{dlev_rsblock}
X=[x_{\ii\jj}]_{\ii,\jj=\bu}^{\nn}=\sum_{\ii,\jj=\bu}^\nn E_{\ii\jj}^{(\nn)}\otimes x_{\ii\jj},
\end{equation}
where $\otimes$ denotes the tensor product; see \eqref{tp:def}.

Two fundamental examples of multilevel block matrices are given by multilevel block Toeplitz matrices and multilevel block diagonal sampling matrices. We provide below the corresponding definitions.

\begin{definition}[multilevel block Toeplitz matrix]\label{mbTm}
Let $f:[-\pi,\pi]^d\to\mathbb C^{r\times s}$ be in $L^1([-\pi,\pi]^d)$ and let $\{f_\kk\}_{\kk\in\mathbb Z^d}$ be the Fourier coefficients of $f$ defined as follows:
\[ f_\kk=\frac1{(2\pi)^d}\int_{[-\pi,\pi]^d}f(\btheta)\,{\rm e}^{-{\rm i}\kk\cdot\btheta}{\rm d}\btheta\ \in\ \mathbb C^{r\times s},\quad\kk\in\mathbb Z^d, \]
where $\kk\cdot\btheta=k_1\theta_1+\ldots+k_d\theta_d$ and the integrals are computed componentwise. For every $\nn\in\mathbb N^d$, the $\nn$th ($d$-level $(r,s)$-block) Toeplitz matrix generated by $f$ is the $d$-level $(r,s)$-block matrix defined as
\[ T_\nn(f)=[f_{\ii-\jj}]_{\ii,\jj=\bu}^\nn=\sum_{\ii,\jj=\bu}^\nn E_{\ii\jj}^{(\nn)}\otimes f_{\ii-\jj}. \]
\end{definition}

\begin{definition}[multilevel block diagonal sampling matrix]\label{mbdsm}
Let $a:[0,1]^d\to\mathbb C^{r\times s}$. For every $\nn\in\mathbb N^d$, the $\nn$th ($d$-level $(r,s)$-block) diagonal sampling matrix generated by $a$ is the $d$-level $(r,s)$-block diagonal matrix defined as
\[ D_\nn(a)=\mathop{\textup{diag}}_{\ii=\bu,\ldots,\nn}a\Bigl(\frac\ii\nn\Bigr)=\sum_{\ii=\bu}^\nn E_{\ii\ii}^{(\nn)}\otimes a\Bigl(\frac\ii\nn\Bigr). \]
\end{definition}

\subsection{Multilevel block matrix-sequences}
Throughout this paper, a sequence of matrices is a sequence of the form $\{A_n\}_n$, where $n$ varies in some infinite subset of $\mathbb N$ and $A_n$ is a $d_n\times e_n$ matrix such that both $d_n$ and $e_n$ tend to $\infty$ as $n\to\infty$.
A $d$-level $(r,s)$-block matrix-sequence is a special sequence of matrices of the form $\{A_\nn\}_n$, where:
\begin{itemize}[leftmargin=*]
	\item $n$ varies in some infinite subset of $\mathbb N$;
	\item	$\nn=\nn(n)$ is a $d$-index in $\mathbb N^d$ which depends on $n$ and satisfies $\nn\to\infty$ as $n\to\infty$;
	\item $A_\nn$ is a matrix of size $N(\nn)r\times N(\nn)s$. 
\end{itemize}
If $\{A_\nn\}_n$ is a $d$-level $(r,s)$-block matrix-sequence, then $A_\nn$ can be seen as a $d$-level $(r,s)$-block matrix 
and can be written in block form as in \eqref{dlev_rsblock}:
\begin{equation*}
A_\nn=[a_{\ii\jj}^{(\nn)}]_{\ii,\jj=\bu}^{\nn}=\sum_{\ii,\jj=\bu}^{\nn}E_{\ii\jj}^{(\nn)}\otimes a_{\ii\jj}^{(\nn)},
\end{equation*}
where $a_{\ii\jj}^{(\nn)}$ is an $r\times s$ matrix. 
A $d$-level $(s,s)$-block matrix-sequence is also referred to as a $d$-level $s$-block matrix-sequence. 

\subsection{Tensor products}\label{sec:tp}

If $X,Y$ are matrices of any dimension, say $X\in\mathbb C^{m_1\times n_1}$ and $Y\in\mathbb C^{m_2\times n_2}$, the tensor (Kronecker) product of $X$ and $Y$ is the $m_1m_2\times n_1n_2$ matrix defined by
\begin{equation}\label{tp:def}
X\otimes Y=[x_{ij}Y]_{i=1,\ldots,m_1}^{j=1,\ldots,n_1}=\left[\begin{array}{ccc}
x_{11}Y & \cdots & x_{1n_1}Y\\
\vdots & & \vdots\\
x_{m_11}Y & \cdots & x_{m_1n_1}Y
\end{array}\right].
\end{equation}
The properties of tensor products that we need in this paper are collected below. For further properties, we refer the reader to \cite[Section~2.5]{GLT-bookII}; see also \cite[Section~2.2.2]{bgd}.

For all matrices $X,Y,Z$, we have
\begin{align}
\label{tensorT}(X\otimes Y)^T&=X^T\otimes Y^T,\\
\label{tensor_asso}X\otimes(Y\otimes Z)&=(X\otimes Y)\otimes Z.
\end{align}
For all matrices $X,Y,Z$ and all scalars $\alpha,\beta\in\mathbb C$, we have
\begin{equation}\label{tensor_bilin}
\left\{\begin{aligned}
(\alpha X+\beta Y)\otimes Z&=\alpha(X\otimes Z)+\beta(Y\otimes Z),\\
X\otimes(\alpha Y+\beta Z)&=\alpha(X\otimes Y)+\beta(X\otimes Z).
\end{aligned}\right.
\end{equation}
Whenever $X_1,X_2$ are multipliable and $Y_1,Y_2$ are multipliable, we have
\begin{equation}\label{tensorM}
(X_1\otimes Y_1)(X_2\otimes Y_2)=(X_1X_2)\otimes(Y_1Y_2).
\end{equation}
For every $k_1,k_2\in\mathbb N$, let $\zeta=[\zeta(1),\zeta(2),\ldots,\zeta(k_1k_2)]$ be the permutation of $[1,2,\ldots,k_1k_2]$ given by
\begin{align*}
\zeta&=[1,k_2+1,2k_2+1,\ldots,(k_1-1)k_2+1,\\
&\hphantom{=[}\;2,k_2+2,2k_2+2,\ldots,(k_1-1)k_2+2,\\
&\hphantom{=[}\;\ldots\,\ldots\,\ldots,\\
&\hphantom{=[}\;k_2,2k_2,3k_2\ldots,k_1k_2],
\end{align*}
i.e.,
\[ \zeta(i)=((i-1)\,\textup{mod}\,k_1)k_2+\left\lfloor\frac{i-1}{k_1}\right\rfloor+1,\quad i=1,\ldots,k_1k_2, \]
and let $P_{k_1,k_2}$ be the permutation matrix associated with $\zeta$, i.e., the $k_1k_2\times k_1k_2$ matrix whose rows are $(\ee_{\zeta(1)}^{(k_1k_2)})^T,\ldots,(\ee_{\zeta(k_1k_2)}^{(k_1k_2)})^T$ (in this order). 
Then,
\begin{equation}\label{Pk1k2}
P_{k_1,k_2}=\left[\begin{array}{c}
I_{k_1} \otimes (\ee_1^{(k_2)})^T \\[3pt]
I_{k_1} \otimes (\ee_2^{(k_2)})^T \\[3pt]
\vdots \\[3pt]
I_{k_1} \otimes (\ee_{k_2}^{(k_2)})^T
\end{array}\right]=\sum_{i=1}^{k_2}\ee_i^{(k_2)}\otimes I_{k_1}\otimes(\ee_i^{(k_2)})^T,
\end{equation}
and
\begin{equation}\label{tensorP}
Y\otimes X=P_{m_1,m_2}(X\otimes Y)P_{n_1,n_2}^T
\end{equation}
for all matrices $X\in\mathbb C^{m_1\times n_1}$ and $Y\in\mathbb C^{m_2\times n_2}$.

\subsection{Singular value and spectral distribution of a sequence of matrices}\label{sv-eig-d}\hfill


\begin{definition}[singular value and spectral distribution of a sequence of matrices]\label{def-distribution}
\begin{itemize}[leftmargin=*]
	\item Let $\{A_n\}_n$ be a sequence of matrices with $A_n$ of size $d_n\times e_n$, and let $f:D\subset\mathbb R^k\to\mathbb C^{r\times s}$ be measurable with $0<\mu_k(D)<\infty$. We say that $\{A_n\}_n$ has a (asymptotic) singular value distribution described by 
	$f$, and we write $\{A_n\}_n\sim_\sigma f$, if
\begin{equation*}
\lim_{n\rightarrow\infty}\frac1{d_n\wedge e_n}\sum_{i=1}^{d_n\wedge e_n}F(\sigma_i(A_n))= 
\frac1{\mu_k(D)}\int_D\frac{\sum_{i=1}^{r\wedge s}F(\sigma_i(f(\mathbf x)))}{r\wedge s}{{\rm d}}\mathbf x,\quad
\forall\,F\in C_c(\mathbb R).
\end{equation*}
In this case, the function $f$ is referred to as the singular value symbol of $\{A_n\}_n$.
	\item Let $\{A_n\}_n$ be a sequence of matrices with $A_n$ of size $d_n\times d_n$, and let $f:D\subset\mathbb R^k\to\mathbb C^{s\times s}$ be measurable with $0<\mu_k(D)<\infty$. We say that $\{A_n\}_n$ has a (asymptotic) spectral (or eigenvalue) distribution described by 
	$f$, and we write $\{A_n\}_n\sim_\lambda f$, if
\begin{equation*}
\lim_{n\rightarrow\infty}\frac1{d_n}\sum_{i=1}^{d_n}F(\lambda_i(A_n))= 
\frac1{\mu_k(D)}\int_D\frac{\sum_{i=1}^{s}F(\lambda_i(f(\mathbf x)))}{s}{{\rm d}}\mathbf x,\quad
\forall\,F\in C_c(\mathbb C).
\end{equation*}
In this case, the function $f$ is referred to as the spectral (or eigenvalue) symbol of $\{A_n\}_n$.
\end{itemize}
\end{definition}

Note that Definition~\ref{def-distribution} is well-posed by \cite[Lemma~2.5]{bgd}, which ensures that the functions $\mathbf x\mapsto\sum_{i=1}^{r\wedge s}F(\sigma_i(f(\mathbf x)))$ and $\mathbf x\mapsto\sum_{i=1}^{s}F(\lambda_i(f(\mathbf x)))$ are measurable. We refer the reader to \cite[Remark~2.9]{bg} for the informal meaning behind the singular value and spectral distribution of a sequence of matrices.
The next lemma will be used (only) in the proof of Theorem~\ref{GLT4-parte2}.

\begin{lemma}\label{sv-easy-version}
Let $\{A_n\}_n$ be a sequence of matrices with $A_n$ of size $d_n\times e_n$, and let $f:D\subset\mathbb R^k\to\mathbb C^{r\times s}$ be measurable with $0<\mu_k(D)<\infty$. If $\{A_n\}_n\sim_\sigma f$ and $f$ has full rank a.e.\ then
\[ \lim_{n\to\infty}\frac{\#\{i\in\{1,\ldots,d_n\wedge e_n\}:\sigma_i(A_n)=0\}}{d_n\wedge e_n}=0. \]
\end{lemma}

We remark that the set $\{f$ has full rank$\}=\{\sigma_{\min}(f)\ne0\}$ is measurable because the function $\mathbf x\mapsto\sigma_{\min}(f(\mathbf x))$ is measurable by \cite[Lemma~2.5]{bgd}.

\begin{proof}
Suppose that $\{A_n\}_n\sim_\sigma f$. For every $M>0$, take $F_M\in C_c(\mathbb R)$ such that $F_M(y)=1-My$ for $0\le y\le1/M$ and $F_M(y)=0$ for $y\ge1/M$.
Since $F_M(0)=1$ and $F_M$ is a non-negative decreasing function on $[0,\infty)$, for every $M>0$ we have
\begin{align}
\label{take_limM}\frac{\#\{i\in\{1,\ldots,d_n\wedge e_n\}:\sigma_i(A_n)=0\}}{d_n\wedge e_n}&\le\frac{1}{d_n\wedge e_n}\sum_{i=1}^{d_n\wedge e_n}F_M(\sigma_i(A_n))\\
\notag&\xrightarrow{n\to\infty}\frac1{\mu_k(D)}\int_D\frac{\sum_{i=1}^{r\wedge s}F_M(\sigma_i(f(\xx)))}{r\wedge s}{\rm d}\xx\\
\notag&\le\frac1{\mu_k(D)}\int_DF_M(\sigma_{\min}(f(\xx))){\rm d}\xx.
\end{align}
Since $F_M(0)=1$ and $F_M\to0$ pointwise on $(0,\infty)$ as $M\to\infty$, the dominated convergence theorem yields
\begin{align*}
\frac1{\mu_k(D)}\int_DF_M(\sigma_{\min}(f(\xx))){\rm d}\xx\xrightarrow{M\to\infty}\frac{\mu_k\{\sigma_{\min}(f)=0\}}{\mu_k(D)},
\end{align*}
which is equal to 0 by the assumption that $f$ has full rank a.e. By taking first the (upper) limit as $n\to\infty$ and then the limit as $M\to\infty$ in \eqref{take_limM}, we get the thesis.
\end{proof}

%

We conclude this section with the definition of zero-distributed sequences.

\begin{definition}[zero-distributed sequence]\label{0d}
A sequence of matrices $\{Z_n\}_n$, with $Z_n$ of size $d_n\times e_n$, is said to be zero-distributed if $\{Z_n\}_n\sim_\sigma0$, i.e., 
\[ \lim_{n\to\infty}\frac1{d_n\wedge e_n}\sum_{i=1}^{d_n\wedge e_n}F(\sigma_i(Z_n))=F(0),\quad\forall\,F\in C_c(\mathbb R). \]
Note that, for any $r,s\ge1$, $\{Z_n\}_n\sim_\sigma0$ is equivalent to $\{Z_n\}_n\sim_\sigma O_{r,s}$.
\end{definition}

\subsection[Rectangular a.c.s.]{Rectangular a.c.s}
The notion of (square) approximating class of sequences (a.c.s.)\ plays a central role in the theory of GLT sequences and has been investigated in \cite{bg,bgd,GLT-bookI,GLT-bookII}; see also \cite{GiovanniLAA2017,ELA}. We here introduce the notion of a.c.s.\ for sequences of rectangular matrices.

\begin{definition}[rectangular a.c.s.]\label{a.c.s.}
Let $\{A_n\}_n$ be a sequence of matrices with $A_n$ of size $d_n\times e_n$, and let $\{\{B_{n,m}\}_n\}_m$ be a sequence of sequences of matrices with $B_{n,m}$ of size $d_n\times e_n$. We say that $\{\{B_{n,m}\}_n\}_m$ is an a.c.s.\ for $\{A_n\}_n$, and we write $\{B_{n,m}\}_n\xrightarrow{\textup{a.c.s.}}\{A_n\}_n$, if the following condition is met: for every $m$ there exists $n_m$ such that, for $n\ge n_m$,
\begin{equation*}
A_n=B_{n,m}+R_{n,m}+N_{n,m},\quad\textup{rank}(R_{n,m})\le c(m)(d_n\wedge e_n),\quad \|N_{n,m}\|\le\omega(m),
\end{equation*}
where $n_m,\,c(m),\,\omega(m)$ depend only on $m$ and $\lim_{m\to\infty}c(m)=\lim_{m\to\infty}\omega(m)=0$.
\end{definition}

In the case where $d_n=e_n$, Definition~\ref{a.c.s.} reduces to the definition of classical square a.c.s.\ \cite[Definition~2.31]{bgd}.

\subsection{GLT sequences}\label{sec:sGLT}

We summarize in this section the theory of square (multilevel block) GLT sequences, which is the basis for the theory of rectangular (multilevel block) GLT sequences developed in this paper. The content of this section can be found in \cite{bgd}.

A $d$-level $s$-block GLT sequence $\{A_\nn\}_n$ is a special $d$-level $s$-block matrix-sequence equipped with a measurable function $\kappa:[0,1]^d\times[-\pi,\pi]^d\to\mathbb C^{s\times s}$, the so-called symbol (or kernel). We use the notation $\{A_\nn\}_n\sim_{\textup{GLT}}\kappa$ to indicate that $\{A_\nn\}_n$ is a $d$-level $s$-block GLT sequence with symbol $\kappa$. 
\begin{enumerate}[leftmargin=35pt]
	\item[\textbf{GLT\,0.}] If $\{A_\nn\}_n\sim_{\textup{GLT}}\kappa$ then $\{A_\nn\}_n\sim_{\textup{GLT}}\xi$ if and only if $\kappa=\xi$ a.e.
	
	If $\kappa:[0,1]^d\times[-\pi,\pi]^d\to\mathbb C^{s\times s}$ is measurable and $\{\nn=\nn(n)\}_n$ is a sequence of $d$-indices such that $\nn\to\infty$ as $n\to\infty$ then there exists $\{A_\nn\}_n\sim_{\textup{GLT}}\kappa$. 
	\item[\textbf{GLT\,1.}] If $\{A_\nn\}_n\sim_{\textup{GLT}}\kappa$ then $\{A_\nn\}_n\sim_\sigma\kappa$. If $\{A_\nn\}_n\sim_{\textup{GLT}}\kappa$ and the matrices $A_\nn$ are Hermitian then $\kappa$ is Hermitian a.e.\ and $\{A_\nn\}_n\sim_\lambda\kappa$.
	\item[\textbf{GLT\,2.}] If $\{A_\nn\}_n\sim_{\textup{GLT}}\kappa$ and $A_\nn=X_\nn+Y_\nn$, where
	\begin{itemize}[leftmargin=*]
		\item every $X_\nn$ is Hermitian,
		\item $(N(\nn))^{-1/2}\|Y_\nn\|_2\to0$,
	\end{itemize}
	then $\{P_\nn^*A_\nn P_\nn\}_n\sim_{\sigma,\lambda}\kappa$ for every sequence $\{P_\nn\}_n$ such that $P_\nn\in\mathbb C^{N(\nn)s\times\delta_n}$, $P_\nn^*P_\nn=I_{\delta_n}$, $\delta_n\le N(\nn)s$, and $\delta_n/(N(\nn)s)\to1$.
	\item[\textbf{GLT\,3.}] For every sequence of $d$-indices $\{\nn=\nn(n)\}_n$ such that $\nn\to\infty$ as $n\to\infty$,
	\begin{itemize}[leftmargin=*]
		\item $\{T_\nn(f)\}_n\sim_{\textup{GLT}}\kappa(\xx,\btheta)=f(\btheta)$ if $f:[-\pi,\pi]^d\to\mathbb C^{s\times s}$ is in $L^1([-\pi,\pi]^d)$,
		\item $\{D_\nn(a)\}_n\sim_{\textup{GLT}}\kappa(\xx,\btheta)=a(\xx)$ if $a:[0,1]^d\to\mathbb C^{s\times s}$ is continuous a.e., 
		\item $\{Z_\nn\}_n\sim_{\textup{GLT}}\kappa(\xx,\btheta)=O_s$ if and only if $\{Z_\nn\}_n\sim_\sigma0$.
	\end{itemize}
	\item[\textbf{GLT\,4.}] If $\{A_\nn\}_n\sim_{\textup{GLT}}\kappa$ and $\{B_\nn\}_n\sim_{\textup{GLT}}\xi$ then
	\begin{itemize}[leftmargin=*]
		\item $\{A_\nn^*\}_n\sim_{\textup{GLT}}\kappa^*$,
		\item $\{\alpha A_\nn+\beta B_\nn\}_n\sim_{\textup{GLT}}\alpha\kappa+\beta\xi$ for all $\alpha,\beta\in\mathbb C$,
		\item $\{A_\nn B_\nn\}_n\sim_{\textup{GLT}}\kappa\xi$,
		\item $\{A_\nn^\dag\}_n\sim_{\textup{GLT}}\kappa^{-1}$ if $\kappa$ is invertible a.e.
	\end{itemize}
	\item[\textbf{GLT\,5.}] If $\{A_\nn\}_n\sim_{\textup{GLT}}\kappa$ and each $A_\nn$ is Hermitian, then $\{f(A_\nn)\}_n\sim_{\textup{GLT}}f(\kappa)$ for every continuous function $f:\mathbb C\to\mathbb C$.
	\item[\textbf{GLT\,6.}] If $\{A_{\nn,ij}\}_n$ is a $d$-level $s$-block GLT sequence with symbol $\kappa_{ij}$ for $i,j=1,\ldots,r$ and $A_\nn=[A_{\nn,ij}]_{i,j=1}^r$ then $\{(P_{r,N(\nn)}\otimes I_s)A_\nn(P_{r,N(\nn)}\otimes I_s)^T\}_n$ is a $d$-level $rs$-block GLT sequence with symbol $\kappa=[\kappa_{ij}]_{i,j=1}^r$, where $P_{k_1,k_2}$ is the permutation matrix defined in \eqref{Pk1k2}.
	\item[\textbf{GLT\,7.}] $\{A_\nn\}_n\sim_{\textup{GLT}}\kappa$ if and only if there exist 
	$\{B_{\nn,m}\}_n\sim_{\textup{GLT}}\kappa_m$ such that $\{B_{\nn,m}\}_n$\linebreak$\xrightarrow{\textup{a.c.s.}}\{A_\nn\}_n$ and $\kappa_m\to\kappa$ in measure.
	\item[\textbf{GLT\,8.}] Suppose $\{A_\nn\}_n\sim_{\textup{GLT}}\kappa$ and $\{B_{\nn,m}\}_n\sim_{\textup{GLT}}\kappa_m$. Then, $\{B_{\nn,m}\}_n\xrightarrow{\textup{a.c.s.}}\{A_\nn\}_n$ if and only if $\kappa_m\to\kappa$ in measure.
	\item[\textbf{GLT\,9.}] If $\{A_\nn\}_n\sim_{\textup{GLT}}\kappa$ then there exist functions $a_{i,m}$, $f_{i,m}$, $i=1,\ldots,N_m$, such~that
	\begin{itemize}[leftmargin=*]
		\item $a_{i,m}:[0,1]^d\to\mathbb C$ belongs to $C^\infty([0,1]^d)$ and $f_{i,m}$ is a trigonometric monomial in $\{{\rm e}^{{\rm i} \jj\cdot\btheta}E_{\alpha\beta}^{(s)}:\jj\in\mathbb Z^d,\ 1\le\alpha,\beta\le s\}$, 
		\item $\kappa_m(\xx,\btheta)=\sum_{i=1}^{N_m}a_{i,m}(\xx)f_{i,m}(\btheta)\to\kappa(\xx,\btheta)$ a.e., \vspace{3pt}
		\item $\{B_{\nn,m}\}_n=\bigl\{\sum_{i=1}^{N_m}D_\nn(a_{i,m}I_{s})T_\nn(f_{i,m})\bigr\}_n\xrightarrow{\textup{a.c.s.}}\{A_\nn\}_n$.
	\end{itemize}
\end{enumerate}

\section{Extension operator}\label{eo}
We introduce in this section the extension operator, which is essential to relate the theory of rectangular GLT sequences to the theory of square GLT sequences.
We also study some properties of this operator that will be needed later on.

\subsection{Definition of extension operator}
In what follows, if $a,t\in\mathbb N$ and $a\le t$, we denote by $\pi_{a,t}$ the $a\times t$ matrix given by $\pi_{a,t}=[\,I_a\,|\,O\,]$.

\begin{definition}[extension operator]
Let $r,s,t$ be positive integers such that $t\ge r\vee s$.
\begin{itemize}[leftmargin=*]
	\item We define the extension operator $E_{r,s}^t:\mathbb C^{r\times s}\to\mathbb C^{t\times t}$ as the linear operator that extends each $r\times s$ matrix to a larger $t\times t$ matrix by adding zero columns to the right and zero rows below:
	\begin{equation}\label{Erst(x)}
	E_{r,s}^t(x)=\left[\begin{array}{cc}x & O\\ O & O\end{array}\right]=\pi_{r,t}^Tx\pi_{s,t}.
	\end{equation}
	\item With some abuse of notation, we define the extension operator $E_{r,s}^t$ also for multilevel $(r,s)$-block matrices. If
	\[ X=[x_{\ii\jj}]_{\ii,\jj=\bu}^\nn=\sum_{\ii,\jj=\bu}^\nn E_{\ii\jj}^{(\nn)}\otimes x_{\ii\jj} \]
	is a $d$-level $(r,s)$-block matrix, 
	then each ``entry'' $x_{\ii\jj}$ is an $r\times s$ block and we define $E_{r,s}^t(X)$ as the $d$-level $t$-block matrix 
	obtained from $X$ by just adding zero columns to the right and zero rows below each block $x_{\ii\jj}$:
	\begin{equation}\label{Erst(X)}
	E_{r,s}^t(X)=[E_{r,s}^t(x_{\ii\jj})]_{\ii,\jj=\bu}^\nn=\sum_{\ii,\jj=\bu}^\nn E_{\ii\jj}^{(\nn)}\otimes E_{r,s}^t(x_{\ii\jj}).
	\end{equation}
\end{itemize}
In the case where $r=s$, we use the notation $E_s^t$ instead of $E_{s,s}^t$ for simplicity.
\end{definition}

By the properties \eqref{tensorT}, \eqref{tensorM}, \eqref{tensorP} of tensor products, for every $d$-level $(r,s)$-block matrix
\begin{align*}
X&=\sum_{\ii,\jj=\bu}^\nn E_{\ii\jj}^{(\nn)}\otimes x_{\ii\jj}=P_{r,N(\nn)}\Biggl(\sum_{\ii,\jj=\bu}^\nn x_{\ii\jj}\otimes E_{\ii\jj}^{(\nn)}\Biggr)P_{s,N(\nn)}^T
\end{align*}
we have
{\allowdisplaybreaks\begin{align*}
E_{r,s}^t(X)&=\sum_{\ii,\jj=\bu}^\nn E_{\ii\jj}^{(\nn)}\otimes E_{r,s}^t(x_{\ii\jj})=P_{t,N(\nn)}\Biggl(\sum_{\ii,\jj=\bu}^\nn E_{r,s}^t(x_{\ii\jj})\otimes E_{\ii\jj}^{(\nn)}\Biggr)P_{t,N(\nn)}^T\\
&=P_{t,N(\nn)}\Biggl(\sum_{\ii,\jj=\bu}^\nn \pi_{r,t}^Tx_{\ii\jj}\pi_{s,t}\otimes E_{\ii\jj}^{(\nn)}\Biggr)P_{t,N(\nn)}^T\\
&=P_{t,N(\nn)}(\pi_{r,t}^T\otimes I_{N(\nn)})\Biggl(\sum_{\ii,\jj=\bu}^\nn x_{\ii\jj}\otimes E_{\ii\jj}^{(\nn)}\Biggr)(\pi_{s,t}\otimes I_{N(\nn)})P_{t,N(\nn)}^T\\
&=P_{t,N(\nn)}\left[\begin{array}{cc}\sum_{\ii,\jj=\bu}^\nn x_{\ii\jj}\otimes E_{\ii\jj}^{(\nn)} & O\\ O & O\end{array}\right]P_{t,N(\nn)}^T\\
&=P_{t,N(\nn)}\left[\begin{array}{cc}P_{r,N(\nn)}^TXP_{s,N(\nn)} & O\\ O & O\end{array}\right]P_{t,N(\nn)}^T,
\end{align*}}%
i.e.,
\begin{align}\label{crucial_formula}
E_{r,s}^t(X)&=Q_{r,t,N(\nn)}\left[\begin{array}{cc}X & O\\ O & O\end{array}\right]Q_{s,t,N(\nn)}^T,
\end{align}
where
\[ Q_{a,t,N(\nn)}=P_{t,N(\nn)}\left[\begin{array}{cc}P_{a,N(\nn)}^T & O\\ O & I_{N(\nn)(t-a)}\end{array}\right] \]
is an $N(\nn)t\times N(\nn)t$ permutation matrix for any $a\in\mathbb N$ with $a\le t$. Equation~\eqref{crucial_formula} can be seen as a definition of $E_{r,s}^t(X)$ alternative to \eqref{Erst(X)}.


\subsection{Algebraic properties}
As it is clear from \eqref{Erst(x)} and \eqref{crucial_formula}, the extension operator $E_{r,s}^t$ is linear on both $\mathbb C^{r\times s}$ and the space 
of $d$-level $(r,s)$-block matrices with fixed level orders $\nn$. Moreover, $E_{r,s}^t$ changes neither the rank nor the norm of the $d$-level $(r,s)$-block matrix $X$ to which it is applied:
\begin{align*}
\textup{rank}(E_{r,s}^t(X))&=\textup{rank}(X),\\
\|E_{r,s}^t(X)\|&=\|X\|.
\end{align*}
If $x\in\mathbb C^{r\times s}$ and $X=[x_{\ii\jj}]_{\ii,\jj=\bu}^\nn$ is a $d$-level $(r,s)$-block matrix then, for every $t\ge r\vee s$,
\begin{align}
\label{Ersx*}(E_{r,s}^t(x))^*&=(\pi_{r,t}^Tx\pi_{s,t})^*=\pi_{s,t}^Tx^*\pi_{r,t}=E_{s,r}^t(x^*),\\
\label{ErsX*}(E_{r,s}^t(X))^*&=\biggl(Q_{r,t,N(\nn)}\left[\begin{array}{cc}X & O\\ O & O\end{array}\right]Q_{s,t,N(\nn)}^T\biggr)^*\\
\notag&=Q_{s,t,N(\nn)}\left[\begin{array}{cc}X^* & O\\ O & O\end{array}\right]Q_{r,t,N(\nn)}^T=E_{s,r}^t(X^*).
\end{align}
If $u\ge t\ge r\vee s$ then, for every $x\in\mathbb C^{r\times s}$,
\begin{align}\label{EuEtx}
E_t^u(E_{r,s}^t(x))=E_{r,s}^u(x).
\end{align}
If $u\ge t\ge r\vee s$ then, for every $d$-level $(r,s)$-block matrix $X=[x_{\ii\jj}]_{\ii,\jj=\bu}^\nn$,
\begin{align}\label{EuEt}
E_t^u(E_{r,s}^t(X))&=E_t^u([E_{r,s}^t(x_{\ii\jj})]_{\ii,\jj=\bu}^\nn)=[E_t^u(E_{r,s}^t(x_{\ii\jj}))]_{\ii,\jj=\bu}^\nn\\
\notag&=[E_{r,s}^u(x_{\ii\jj})]_{\ii,\jj=\bu}^\nn=E_{r,s}^u(X).
\end{align}
If $x\in\mathbb C^{r\times q}$ and $y\in\mathbb C^{q\times s}$ then, 
for every $t\ge r\vee q\vee s$, 
\begin{equation}\label{E(xy)}
E_{r,s}^t(xy)=\pi_{r,t}^Txy\pi_{s,t}=\pi_{r,t}^Tx\pi_{q,t}\pi_{q,t}^Ty\pi_{s,t}=E_{r,q}^t(x)E_{q,s}^t(y).
\end{equation}
If $X=[x_{\ii\jj}]_{\ii,\jj=\bu}^\nn$ and $Y=[y_{\ii\jj}]_{\ii,\jj=\bu}^\nn$, with $x_{\ii\jj}\in\mathbb C^{r\times q}$ and $y_{\ii\jj}\in\mathbb C^{q\times s}$ for all $\ii,\jj=\bu,\ldots,\nn$, then, for every $t\ge r\vee q\vee s$,
\begin{align}
\label{ah}E_{r,s}^t(XY)&=Q_{r,t,N(\nn)}\left[\begin{array}{cc}XY & O\\ O & O\end{array}\right]Q_{s,t,N(\nn)}^T\\
\notag&=Q_{r,t,N(\nn)}\left[\begin{array}{cc}X & O\\ O & O\end{array}\right]Q_{q,t,N(\nn)}^TQ_{q,t,N(\nn)}\left[\begin{array}{cc}Y & O\\ O & O\end{array}\right]Q_{s,t,N(\nn)}^T\\
\notag&=E_{r,q}^t(X)E_{q,s}^t(Y).
\end{align}

\subsection{Singular value distribution of extended matrix-sequences}\hfill
\begin{proposition}\label{sigmaP}
Let $\{A_\nn\}_n$ be a $d$-level $(r,s)$-block matrix-sequence and let $f:D\subset\mathbb R^k\to\mathbb C^{r\times s}$ be measurable with $0<\mu_k(D)<\infty$. For any $t\ge r\vee s$ we have
\[ \{A_\nn\}_n\sim_\sigma f\ \iff\ \{E_{r,s}^t(A_\nn)\}_n\sim_\sigma E_{r,s}^t(f). \]
\end{proposition}
\begin{proof}
Let $\ell=r\wedge s$. For every $\xx\in D$,
\begin{alignat*}{3}
\sigma_i(E_{r,s}^t(f(\xx)))&=\sigma_i(f(\xx)), &\quad i&=1,\ldots,\ell,\\
\sigma_i(E_{r,s}^t(f(\xx)))&=0, &\quad i&=\ell+1,\ldots,t.
\end{alignat*}
Moreover, by \eqref{crucial_formula},
\begin{alignat*}{3}
\sigma_i(E_{r,s}^t(A_\nn))&=\sigma_i(A_\nn), &\quad  i&=1,\ldots, N(\nn)\ell,\\
\sigma_i(E_{r,s}^t(A_\nn))&=0, &\quad i&=N(\nn)\ell+1,\ldots,N(\nn)t.
\end{alignat*}
Thus, for every $F\in C_c(\mathbb R)$,
\begin{align*}
\frac1{N(\nn)t}\sum_{i=1}^{N(\nn)t}F(\sigma_i(E_{r,s}^t(A_\nn)))&=\frac\ell t\,\frac1{N(\nn)\ell}\sum_{i=1}^{N(\nn)\ell}F(\sigma_i(A_\nn))+\frac{t-\ell}t\,F(0),\\
\int_D\frac{\sum_{i=1}^{t}F(\sigma_i(E_{r,s}^t(f(\xx))))}t\,{\rm d}\mathbf x &=	\frac \ell t\int_D\frac{\sum_{i=1}^{\ell}F(\sigma_i(f(\xx)))}{\ell}\,{\rm d}\mathbf x+\mu_k(D)\,\frac{t-\ell}t\,F(0).
\end{align*}
Therefore, $\{A_\nn\}_n\sim_\sigma f$ if and only if $\{E_{r,s}^t(A_\nn)\}_n\sim_\sigma E_{r,s}^t(f)$.
\end{proof}

\subsection[Extended a.c.s.]{Extended a.c.s}\hfill

\begin{proposition}\label{e.a.c.s.}
Let $\{A_\nn\}_n$ and $\{B_{\nn,m}\}_n$ be $d$-level $(r,s)$-block matrix-sequences. For any $t\ge r\vee s$ we have
\[ \{B_{\nn,m}\}_n\xrightarrow{\textup{a.c.s.}}\{A_\nn\}_n\ \iff\ \{E_{r,s}^t(B_{\nn,m})\}_n\xrightarrow{\textup{a.c.s.}}\{E_{r,s}^t(A_\nn)\}_n. \]
\end{proposition}
\begin{proof}
($\implies$) Suppose that $\{B_{\nn,m}\}_n\xrightarrow{\textup{a.c.s.}}\{A_\nn\}_n$. Then, for every $m$ there exists $n_m$ such that, for $n\ge n_m$,
\[ A_\nn=B_{\nn,m}+R_{\nn,m}+N_{\nn,m},\quad\textup{rank}(R_{\nn,m})\le c(m)N(\nn),\quad\|N_{\nn,m}\|\le\omega(m), \]
where $c(m),\omega(m)\to0$ as $m\to\infty$. By applying the extension operator to both sides of the previous equation, we obtain
\[ E_{r,s}^t(A_\nn)=E_{r,s}^t(B_{\nn,m})+E_{r,s}^t(R_{\nn,m})+E_{r,s}^t(N_{\nn,m}) \]
with
\begin{align*}
\textup{rank}(E_{r,s}^t(R_{\nn,m}))&=\textup{rank}(R_{\nn,m})\le c(m)N(\nn),\\
\|E_{r,s}^t(N_{\nn,m})\|&=\|N_{\nn,m}\|\le\omega(m).
\end{align*}
This shows that $\{E_{r,s}^t(B_{\nn,m})\}_n\xrightarrow{\textup{a.c.s.}}\{E_{r,s}^t(A_\nn)\}_n$.

($\impliedby$) Suppose that $\{E_{r,s}^t(B_{\nn,m})\}_n\xrightarrow{\textup{a.c.s.}}\{E_{r,s}^t(A_\nn)\}_n$. Then, for every $m$ there exists $n_m$ such that, for $n\ge n_m$,
\[ E_{r,s}^t(A_\nn)=E_{r,s}^t(B_{\nn,m})+R_{\nn,m}+N_{\nn,m},\ \ \,\textup{rank}(R_{\nn,m})\le c(m)N(\nn),\ \ \,\|N_{\nn,m}\|\le\omega(m), \]
where $c(m),\omega(m)\to0$ as $m\to\infty$. By \eqref{crucial_formula}, the previous equation is equivalent to
\[ \left[\begin{array}{cc}A_\nn & O\\ O & O\end{array}\right]=\left[\begin{array}{cc}B_{\nn,m} & O\\ O & O\end{array}\right]+Q_{r,t,N(\nn)}^TR_{\nn,m}Q_{s,t,N(\nn)}+Q_{r,t,N(\nn)}^TN_{\nn,m}Q_{s,t,N(\nn)}. \]
This implies that
\begin{align*}
A_\nn&=B_{\nn,m}+\Pi_{r,t,N(\nn)}Q_{r,t,N(\nn)}^TR_{\nn,m}Q_{s,t,N(\nn)}\Pi_{s,t,N(\nn)}^T\\
&\hphantom{{}={}}+\Pi_{r,t,N(\nn)}Q_{r,t,N(\nn)}^TN_{\nn,m}Q_{s,t,N(\nn)}\Pi_{s,t,N(\nn)}^T,
\end{align*}
where $\Pi_{a,t,N(\nn)}$ is the $N(\nn)a\times N(\nn)t$ matrix given by $\Pi_{a,t,N(\nn)}=[\,I_{N(\nn)a}\,|\,O\,]$ for every $a\in\mathbb N$ with $a\le t$. Since $\|\Pi_{a,t,N(\nn)}\|=\|Q_{a,t,N(\nn)}\|=1$, we have
\begin{align*}
\textup{rank}(\Pi_{r,t,N(\nn)}Q_{r,t,N(\nn)}^TR_{\nn,m}Q_{s,t,N(\nn)}\Pi_{s,t,N(\nn)}^T)&\le\textup{rank}(R_{\nn,m})\le c(m)N(\nn),\\
\|\Pi_{r,t,N(\nn)}Q_{r,t,N(\nn)}^TN_{\nn,m}Q_{s,t,N(\nn)}\Pi_{s,t,N(\nn)}^T\|&\le\|N_{\nn,m}\|\le\omega(m),
\end{align*}
and we conclude that $\{B_{\nn,m}\}_n\xrightarrow{\textup{a.c.s.}}\{A_\nn\}_n$.
\end{proof}

\subsection{Extended GLT sequences}
Let $a:[0,1]^d\to\mathbb C^{r\times s}$ be continuous a.e.\ on $[0,1]^d$, let $f:[-\pi,\pi]^d\to\mathbb C^{r\times s}$ be in $L^1([-\pi,\pi]^d)$, and take $\nn\in\mathbb N^d$ and $t\ge r\vee s$. Then, by definition of $E_{r,s}^t$,
\begin{align}
\label{Dna-estesa}E_{r,s}^t(D_\nn(a))&=\mathop{\textup{diag}}_{\ii=\bu,\ldots,\nn}E_{r,s}^t\Bigl(a\Bigl(\frac\ii\nn\Bigr)\Bigr)=D_\nn(E_{r,s}^t(a)),\\
\label{Tnf-estesa}E_{r,s}^t(T_\nn(f))&=[E_{r,s}^t(f_{\ii-\jj})]_{\ii,\jj=\bu}^\nn=[(E_{r,s}^t(f))_{\ii-\jj}]_{\ii,\jj=\bu}^\nn=T_\nn(E_{r,s}^t(f)).
\end{align}
In the case where $r=s$, it follows from \eqref{Dna-estesa}--\eqref{Tnf-estesa} and {\bf GLT\,3} that, for every sequence of $d$-indices $\{\nn=\nn(n)\}_n$ such that $\nn\to\infty$ as $n\to\infty$,
\begin{align} 
\label{Dna-estesa-GLT}\{E_s^t(D_\nn(a))\}_n\sim_{\textup{GLT}}E_s^t(a(\xx)),\\
\label{Tnf-estesa-GLT}\{E_s^t(T_\nn(f))\}_n\sim_{\textup{GLT}}E_s^t(f(\btheta)).
\end{align}
Proposition~\ref{GLT-estesa-GLT} generalizes \eqref{Dna-estesa-GLT}--\eqref{Tnf-estesa-GLT} by showing that an extended GLT sequence is still a GLT sequence with symbol given by the extended symbol. For the proof of Proposition~\ref{GLT-estesa-GLT}, we need the following lemma \cite[Lemma~2.4]{GLT-bookII}.

\begin{lemma}\label{l2.4}
Let $\kappa:[0,1]^d\times[-\pi,\pi]^d\to\mathbb C$ be measurable. Then, there exists a sequence of functions $\kappa_m:[0,1]^d\times[-\pi,\pi]^d\to\mathbb C$ such that $\kappa_m\to\kappa$ a.e.\ and $\kappa_m$ is of the form
\begin{equation*}
\kappa_m(\xx,\btheta)=\sum_{\jj=-\boldsymbol N_m}^{\boldsymbol N_m}a_\jj^{(m)}(\xx)\,{\rm e}^{{\rm i}\jj\cdot\btheta},\quad a_\jj^{(m)}\in C^\infty([0,1]^d),\quad \boldsymbol N_m\in\mathbb N^d.
\end{equation*}
\end{lemma}

\begin{proposition}\label{GLT-estesa-GLT}
Let $\{A_\nn\}_n$ be a $d$-level $s$-block matrix-sequence and let $\kappa:[0,1]^d\times[-\pi,\pi]^d\to\mathbb C^{s\times s}$ be measurable. For any $t\ge s$ we have
\[ \{A_\nn\}_n\sim_{\textup{GLT}}\kappa\ \iff\ \{E_s^t(A_\nn)\}_n\sim_{\textup{GLT}}E_s^t(\kappa). \]
\end{proposition}
\begin{proof}
($\implies$) Suppose $\{A_\nn\}_n\sim_{\textup{GLT}}\kappa$. By {\bf GLT\,9}, there exist functions $a_{i,m}$, $f_{i,m}$, $i=1,\ldots,N_m$, such that $a_{i,m}:[0,1]^d\to\mathbb C$ belongs to $C^\infty([0,1]^d)$, $f_{i,m}$ is a trigonometric monomial in $\{{\rm e}^{{\rm i}\jj\cdot\btheta}E_{\alpha\beta}^{(s)}:\jj\in\mathbb Z^d,\ 1\le\alpha,\beta\le s\}$, and
\begin{itemize}[leftmargin=*]
	\item $\kappa_m(\xx,\btheta)=\sum_{i=1}^{N_m}a_{i,m}(\xx)f_{i,m}(\btheta)\to\kappa(\xx,\btheta)$ a.e., \vspace{1pt}
	\item $\{B_{\nn,m}\}_n=\bigl\{\sum_{i=1}^{N_m}D_\nn(a_{i,m}I_{s})T_\nn(f_{i,m})\bigr\}_n\xrightarrow{\textup{a.c.s.}}\{A_\nn\}_n$.
\end{itemize}
By the linearity of $E_s^t$, properties \eqref{E(xy)}--\eqref{ah}, equations \eqref{Dna-estesa-GLT}--\eqref{Tnf-estesa-GLT}, and {\bf GLT\,4},
\begin{align}\label{GLT7a}
\{E_s^t(B_{\nn,m})\}_n&=\Biggl\{\sum_{i=1}^{N_m}E_s^t(D_\nn(a_{i,m}I_s))E_s^t(T_\nn(f_{i,m}))\Biggr\}_n\\
\notag&\sim_{\textup{GLT}}\sum_{i=1}^{N_m}E_s^t(a_{i,m}(\xx)I_s)E_s^t(f_{i,m}(\btheta))=E_s^t(\kappa_m(\xx,\btheta)).
\end{align}
By Proposition \ref{e.a.c.s.},
\begin{equation}\label{GLT7b}
\{E_s^t(B_{\nn,m})\}_n\xrightarrow{\textup{a.c.s.}}\{E_s^t(A_\nn)\}_n.
\end{equation}
Finally, it is clear that
\begin{equation}\label{GLT7c}
E_s^t(\kappa_m(\xx,\btheta))\to E_s^t(\kappa(\xx,\btheta))\mbox{ \,a.e.}
\end{equation}
Equations \eqref{GLT7a}--\eqref{GLT7c} and {\bf GLT\,7} yield the thesis $\{E_s^t(A_\nn)\}_n\sim_{\textup{GLT}}E_s^t(\kappa)$.

($\impliedby$) Suppose $\{E_s^t(A_\nn)\}_n\sim_{\textup{GLT}}E_s^t(\kappa)$. Let $a_{i,m}$, $f_{i,m}$, $i=1,\ldots,N_m$, be functions such that $a_{i,m}:[0,1]^d\to\mathbb C$ is continuous a.e., $f_{i,m}:[-\pi,\pi]^d\to\mathbb C^{s\times s}$ is in $L^1([-\pi,\pi]^d)$, and
\begin{equation}\label{GLT7.1}
\kappa_m(\xx,\btheta)=\sum_{i=1}^{N_m}a_{i,m}(\xx)f_{i,m}(\btheta)\to\kappa(\xx,\btheta)\mbox{ \,a.e.}
\end{equation}
These functions exist by Lemma~\ref{l2.4}, which in fact ensures we can take $a_{i,m}\in C^\infty([0,1]^d)$ and $f_{i,m}$ in the set of trigonometric monomials $\{{\rm e}^{{\rm i}\jj\cdot\btheta}E_{\alpha\beta}^{(s)}:\jj\in\mathbb Z^d,\ 1\le\alpha,\beta\le s\}$.
Let
\[ B_{\nn,m}=\sum_{i=1}^{N_m}D_\nn(a_{i,m}I_{s})T_\nn(f_{i,m}) \]
and note that, by {\bf GLT\,3}\,--\,{\bf GLT\,4},
\begin{equation}\label{GLT7.2}
\{B_{\nn,m}\}_n\sim_{\textup{GLT}}\kappa_m(\xx,\btheta).
\end{equation}
We have $\{E_s^t(B_{\nn,m})\}_n\sim_{\textup{GLT}}E_s^t(\kappa_m(\xx,\btheta))$ (see \eqref{GLT7a}) and $E_s^t(\kappa_m(\xx,\btheta))\to E_s^t(\kappa(\xx,\btheta))$ a.e.\ (by \eqref{GLT7.1}). Keeping in mind the assumption $\{E_s^t(A_\nn)\}_n\sim_{\textup{GLT}}E_s^t(\kappa(\xx,\btheta))$ and using {\bf GLT\,8}, we obtain
\begin{equation*}
\{E_s^t(B_{\nn,m})\}_n\xrightarrow{\textup{a.c.s.}}\{E_s^t(A_\nn)\}_n.
\end{equation*}
By Proposition~\ref{e.a.c.s.}, this implies that
\begin{equation}\label{GLT7.3}
\{B_{\nn,m}\}_n\xrightarrow{\textup{a.c.s.}}\{A_\nn\}_n.
\end{equation}
Equations \eqref{GLT7.1}--\eqref{GLT7.3} and {\bf GLT\,7} yield the thesis $\{A_\nn\}_n\sim_{\textup{GLT}}\kappa$.
\end{proof}

\section{Rectangular GLT sequences}\label{sec:GLT}
We develop in this section the theory of rectangular (multilevel block) GLT sequences as an extension of the theory of square (multilevel block) GLT sequences.
The key tool for transferring results about square GLT sequences to rectangular GLT sequences is the extension operator studied in Section~\ref{eo}.

\subsection{Definition of rectangular GLT sequences}\hfill

\begin{definition}[rectangular GLT sequence]\label{rGLT}
Let $\{A_\nn\}_n$ be a $d$-level $(r,s)$-block matrix-sequence and let $\kappa:[0,1]^d\times[-\pi,\pi]^d\to\mathbb C^{r\times s}$ be measurable. We say that $\{A_\nn\}_n$ is a ($d$-level $(r,s)$-block) GLT sequence with symbol $\kappa$, and we write $\{A_\nn\}_n\sim_{\textup{GLT}}\kappa$, if one of the following equivalent conditions is satisfied.
\begin{enumerate}[leftmargin=*]
	\item $\{E_{r,s}^t(A_\nn)\}_n\sim_{\textup{GLT}}E_{r,s}^t(\kappa)$ for all $t\ge r\vee s$.
	\item There exists $t\ge r\vee s$ such that $\{E_{r,s}^t(A_\nn)\}_n\sim_{\textup{GLT}}E_{r,s}^t(\kappa)$.
\end{enumerate}
\end{definition}
\begin{proof}
We prove the equivalence between the two conditions in Definition~\ref{rGLT}.

($1\implies2$) Obvious.

($2\implies1$) Suppose there exists $t\ge r\vee s$ such that $\{E_{r,s}^t(A_\nn)\}_n\sim_{\textup{GLT}}E_{r,s}^t(\kappa)$. We show that $\{E_{r,s}^u(A_\nn)\}_n\sim_{\textup{GLT}}E_{r,s}^u(\kappa)$ for all $u\ge r\vee s$.
If $u\ge t$ then, by \eqref{EuEtx}--\eqref{EuEt} and Proposition~\ref{GLT-estesa-GLT},
\[ \{E_{r,s}^u(A_\nn)\}_n=\{E_t^u(E_{r,s}^t(A_\nn))\}_n\sim_{\textup{GLT}}E_t^u(E_{r,s}^t(\kappa))=E_{r,s}^u(\kappa). \]
If $r\vee s\le u\le t$ then, by \eqref{EuEtx}--\eqref{EuEt},
\[ \{E_u^t(E_{r,s}^u(A_\nn))\}_n=\{E_{r,s}^t(A_\nn)\}_n\sim_{\textup{GLT}}E_{r,s}^t(\kappa)=E_u^t(E_{r,s}^u(\kappa)), \]
which implies that $\{E_{r,s}^u(A_\nn)\}_n\sim_{\textup{GLT}}E_{r,s}^u(\kappa)$ by Proposition~\ref{GLT-estesa-GLT}.
\end{proof}

\begin{remark}
Definition~\ref{rGLT} is consistent with the definition of multilevel block GLT sequences given in \cite{bgd}. Indeed, let $\{A_\nn\}_n$ be a $d$-level $s$-block matrix-sequence and let $\kappa:[0,1]^d\times[-\pi,\pi]^d\to\mathbb C^{s\times s}$ be measurable. Then, $\{A_\nn\}_n\sim_{\textup{GLT}}\kappa$ according to the definition in \cite{bgd} if and only if $\{A_\nn\}_n\sim_{\textup{GLT}}\kappa$ according to Definition~\ref{rGLT}; see Proposition~\ref{GLT-estesa-GLT}.
\end{remark}

According to Definition~\ref{rGLT}, the extension operator ``embeds'' the world of rectangular GLT sequences in the world of square GLT sequences. 
As we shall see in the next sections, this embedding allows us to transfer most of the properties {\bf GLT\,0}\,--\,{\bf GLT\,9} to rectangular GLT sequences. Note, however, that we cannot transfer to rectangular GLT sequences the properties that involve spectral symbols or Hermitian matrices.

\subsection{Uniqueness of the symbol of a rectangular GLT sequence}
The next theorem proves the analog of the first part of {\bf GLT\,0} for rectangular GLT sequences.

\begin{theorem}\label{GLT0-parte1}
Let $\{A_\nn\}_n$ 
be a $d$-level $(r,s)$-block GLT sequence with symbol $\kappa$ and let $\xi:[0,1]^d\times[-\pi,\pi]^d\to\mathbb C^{r\times s}$ be measurable. Then,
\[ \{A_\nn\}_n\sim_{\textup{GLT}}\xi\ \iff\ \kappa=\xi\mbox{ \,a.e.\ in }\,[0,1]^d\times[-\pi,\pi]^d. \]
\end{theorem}
\begin{proof}
($\implies$) Let $t\ge r\vee s$. Since $\{A_\nn\}_n\sim_{\textup{GLT}}\kappa$ and $\{A_\nn\}_n\sim_{\textup{GLT}}\xi$, we have $\{E_{r,s}^t(A_\nn)\}_n\sim_{\textup{GLT}}E_{r,s}^t(\kappa)$ and $\{E_{r,s}^t(A_\nn)\}_n\sim_{\textup{GLT}}E_{r,s}^t(\xi)$ by Definition~\ref{rGLT}. This implies that $E_{r,s}^t(\kappa)=E_{r,s}^t(\xi)$ a.e.\ by {\bf GLT\,0}, and so $\kappa=\xi$ a.e.

($\impliedby$) Let $t\ge r\vee s$. Since $\{A_\nn\}_n\sim_{\textup{GLT}}\kappa$ and $\kappa=\xi$ a.e., we have $\{E_{r,s}^t(A_\nn)\}_n\sim_{\textup{GLT}}E_{r,s}^t(\kappa)$ by Definition~\ref{rGLT} and $E_{r,s}^t(\kappa)=E_{r,s}^t(\xi)$ a.e. This implies that $\{E_{r,s}^t(A_\nn)\}_n\sim_{\textup{GLT}}E_{r,s}^t(\xi)$ by {\bf GLT\,0}, and so $\{A_\nn\}_n\sim_{\textup{GLT}}\xi$ by Definition~\ref{rGLT}.
\end{proof}

\subsection{Fundamental examples of rectangular GLT sequences}\label{GLT3}
In this section, we prove the analog of {\bf GLT\,3} for rectangular GLT sequences.

\subsubsection{Rectangular Toeplitz sequences}
Let $f:[-\pi,\pi]^d\to\mathbb C^{r\times s}$ be in $L^1([-\pi,\pi]^d)$ and let $\{T_\nn(f)\}_{\nn\in\mathbb N^d}$ be the family of Toeplitz matrices generated by $f$ (see Definition~\ref{mbTm}). By \eqref{Tnf-estesa}, Definition~\ref{rGLT}, and {\bf GLT\,3}, for every sequence of $d$-indices $\{\nn=\nn(n)\}_n$ such that $\nn\to\infty$ as $n\to\infty$, we have $\{T_\nn(f)\}_n\sim_{\textup{GLT}}f(\btheta)$.

\subsubsection{Sequences of rectangular diagonal sampling matrices}
Let $a:[0,1]^d\to\mathbb C^{r\times s}$ be continuous a.e.\ on $[0,1]^d$ and let $\{D_\nn(a)\}_{\nn\in\mathbb N^d}$ be the family of diagonal sampling matrices generated by $a$ (see Definition~\ref{mbdsm}). By \eqref{Dna-estesa}, Definition~\ref{rGLT}, and {\bf GLT\,3}, for every sequence of $d$-indices $\{\nn=\nn(n)\}_n$ such that $\nn\to\infty$ as $n\to\infty$, we have $\{D_\nn(a)\}_n\sim_{\textup{GLT}}a(\xx)$.

\subsubsection{Rectangular zero-distributed sequences}
Suppose that $\{Z_\nn\}_n$ is a $d$-level $(r,s)$-block zero-distributed sequence. Then, $\{E_{r,s}^t(Z_\nn)\}_n$ is zero-distributed for any $t\ge r\vee s$; see Proposition~\ref{sigmaP}. Hence, by Definition~\ref{rGLT}, $\{Z_\nn\}_n\sim_{\textup{GLT}}O_{r,s}$.

\subsection{Singular value distribution of rectangular GLT sequences}
The next theorem proves the analog of {\bf GLT\,1} for rectangular GLT sequences.


\begin{theorem}\label{GLT1}
If $\{A_\nn\}_n\sim_{\textup{GLT}}\kappa$ then $\{A_\nn\}_n\sim_\sigma\kappa$.
\end{theorem}
\begin{proof}
Let $t\ge r\vee s$, where $r\times s$ is the size of $\kappa$.
Since $\{E_{r,s}^t(A_\nn)\}_n\sim_{\textup{GLT}}E_{r,s}^t(\kappa)$ by Definition~\ref{rGLT}, we have $\{E_{r,s}^t(A_\nn)\}_n\sim_\sigma E_{r,s}^t(\kappa)$ by {\bf GLT\,1}, which implies $\{A_\nn\}_n\sim_\sigma\kappa$ by Proposition~\ref{sigmaP}.
\end{proof}

\subsection{Rectangular GLT algebra}
Suppose $\{A_\nn\}_n\sim_{\textup{GLT}}\kappa$ and $\{B_\nn\}_n\sim_{\textup{GLT}}\xi$. If $\kappa$ and $\xi$ are summable, then the same is true for $A_\nn$ and $B_\nn$, and so we can consider the sequence $\{\alpha A_\nn+\beta B_\nn\}_n$ for $\alpha,\beta\in\mathbb C$. Similarly, if $\kappa$ and $\xi$ are multipliable, then the same is true for $A_\nn$ and $B_\nn$, and so we can consider the sequence $\{A_\nn B_\nn\}_n$. The next theorem proves the analog of the first part of {\bf GLT\,4} for rectangular GLT sequences.

\begin{theorem}\label{GLT4-parte1}
Let $\{A_\nn\}_n\sim_{\textup{GLT}}\kappa$ and $\{B_\nn\}_n\sim_{\textup{GLT}}\xi$. Then,
\begin{enumerate}[leftmargin=*]
	\item $\{A_\nn^*\}_n\sim_{\textup{GLT}}\kappa^*$,
	\item $\{\alpha A_\nn+\beta B_\nn\}_n\sim_{\textup{GLT}}\alpha\kappa+\beta\xi$ for all $\alpha,\beta\in\mathbb C$ if $\kappa$ and $\xi$ are summable,
	\item $\{A_\nn B_\nn\}_n\sim_{\textup{GLT}}\kappa\xi$ if $\kappa$ and $\xi$ are multipliable.
\end{enumerate}
\end{theorem}
\begin{proof}
1. Let $t\ge r\vee s$, where $r\times s$ is the size of $\kappa$. By \eqref{Ersx*}--\eqref{ErsX*}, Definition~\ref{rGLT} and {\bf GLT\,4},
\[ \{E_{s,r}^t(A_\nn^*)\}_n=\{(E_{r,s}^t(A_\nn))^*\}_n\sim_{\textup{GLT}}(E_{r,s}^t(\kappa))^*=E_{s,r}^t(\kappa^*). \]
We conclude that $\{A_\nn^*\}_n\sim_{\textup{GLT}}\kappa^*$ by Definition~\ref{rGLT}.

2. Let $t\ge r\vee s$, where $r\times s$ is the size of $\kappa$ and $\xi$. By the linearity of the extension operator, Definition~\ref{rGLT} and {\bf GLT\,4},
\begin{align*}
\{E_{r,s}^t(\alpha A_\nn+\beta B_\nn)\}_n&=\{\alpha E_{r,s}^t(A_\nn)+\beta E_{r,s}^t(B_\nn)\}_n\\
&\sim_{\textup{GLT}}\alpha E_{r,s}^t(\kappa)+\beta E_{r,s}^t(\xi)=E_{r,s}^t(\alpha\kappa+\beta\xi).
\end{align*}
We conclude that $\{\alpha A_\nn+\beta B_\nn\}_n\sim_{\textup{GLT}}\alpha\kappa+\beta\xi$ by Definition~\ref{rGLT}.

3. Let $t\ge r\vee q\vee s$, where $r\times q$ is the size of $\kappa$ and $q\times s$ is the size of $\xi$. By \eqref{E(xy)}--\eqref{ah}, Definition~\ref{rGLT} and {\bf GLT\,4},
\begin{align*}
\{E_{r,s}^t(A_\nn B_\nn)\}_n=\{E_{r,q}^t(A_\nn)E_{q,s}^t(B_\nn)\}_n\sim_{\textup{GLT}}E_{r,q}^t(\kappa)E_{q,s}^t(\xi)=E_{r,s}^t(\kappa\xi).
\end{align*}
We conclude that $\{A_\nn B_\nn\}_n\sim_{\textup{GLT}}\kappa\xi$ by Definition~\ref{rGLT}.
\end{proof}

To prove the analog of the second part of {\bf GLT\,4} for rectangular GLT sequences, we need to recall some properties of the Moore--Penrose pseudoinverse \cite[Section~7.6]{Bini}. If $A=U\Sigma V$ is a singular value decomposition (SVD) of the $m\times n$ matrix $A$, then $A^\dag=V^*\Sigma^\dag U^*$. Here, $\Sigma^\dag$ is the Moore--Penrose pseudoinverse of $\Sigma$, i.e., the $n\times m$ diagonal matrix such that, for $i=1,\ldots,m\wedge n$, $(\Sigma^\dag)_{ii}=1/\Sigma_{ii}$ if $\Sigma_{ii}\ne0$ and $(\Sigma^\dag)_{ii}=0$ otherwise. If $A$ is an $m\times n$ full rank matrix, then $A^\dag$ can be expressed as follows:
\begin{equation}\label{crocino}
A^\dag=\left\{\begin{aligned}
&A^*(AA^*)^{-1}, &&\mbox{if }\,m\le n,\\
&(A^*A)^{-1}A^*, &&\mbox{if }\,m\ge n.
\end{aligned}\right.
\end{equation}
Note that $A^\dag=A^{-1}$ whenever $A$ is a square invertible matrix. Theorem~\ref{GLT4-parte2} proves the analog of the second part of {\bf GLT\,4} for rectangular GLT sequences.

\begin{theorem}\label{GLT4-parte2}
If $\{A_\nn\}_n\sim_{\textup{GLT}}\kappa$ and $\kappa$ has full rank a.e.\ then $\{A_\nn^\dag\}_n\sim_{\textup{GLT}}\kappa^\dag$.
\end{theorem}
\begin{proof}
Let $A_\nn=U_\nn\Sigma_\nn V_\nn$ be an SVD of $A_\nn$ and let $Z_\nn=U_\nn\Psi_\nn V_\nn$, where $\Psi_\nn$ is the rectangular diagonal matrix of the same size as $\Sigma_\nn$ such that $(\Psi_\nn)_{ii}=1$ if $(\Sigma_\nn)_{ii}=0$ and $(\Psi_\nn)_{ii}=0$ otherwise.
The rank of $Z_\nn$ is the number of zero singular values of $A_\nn$, which is $o(N(\nn))$ by Theorem~\ref{GLT1} and Lemma~\ref{sv-easy-version}, since $\kappa$ has full rank a.e. Hence, the sequence $\{Z_\nn\}_n$ is zero-distributed by Definition~\ref{0d}, and so $\{Z_\nn\}_n\sim_{\textup{GLT}}O_{r,s}$, with $r\times s$ being the size of $\kappa$. Let
\begin{equation}\label{partial}
B_\nn = A_\nn+Z_\nn = U_\nn(\Sigma_\nn+\Psi_\nn)V_\nn
\end{equation}
and note that $\{B_\nn\}_n\sim_{\textup{GLT}}\kappa$ by Theorem~\ref{GLT4-parte1}. The matrix $B_\nn$ has full rank by construction and so, by \eqref{partial} and \eqref{crocino},\footnote{We here assume that $r\le s$. If $r\ge s$, nothing changes in the proof except the fact that we have to use for the pseudoinverse $B_\nn^\dag$ the other expression in \eqref{crocino}.}
\[ B_\nn^\dag = A_\nn^\dag + Z_\nn^\dag = B^*_\nn (B_\nn B_\nn^*)^{-1}. \]
The rank of $Z_\nn^\dag$ is the same as the rank of $Z_\nn$, which implies that $\{Z_\nn^\dag\}_n$ is zero-distributed and $\{Z_\nn^\dag\}_n\sim_{\textup{GLT}}O_{s,r}$.
Moreover, $\{B_\nn B_\nn^*\}_n$ is a square GLT sequence with symbol $\kappa\kappa^*$ by Theorem~\ref{GLT4-parte1}, and $\kappa\kappa^*$ is invertible a.e.\ because $\kappa$ has full rank a.e. We can therefore use \textbf{GLT\,4} to obtain
\[ \{(B_\nn B_\nn^*)^{-1}\}_n = \{(B_\nn B_\nn^*)^\dag\}_n \sim_{\textup{GLT}} (\kappa\kappa^*)^{-1}. \]
Using again Theorem~\ref{GLT4-parte1}, we conclude that
\[ \{A_\nn^\dag\}_n = \{B_\nn^\dag-Z_\nn^\dag\}_n = \{B_\nn^*(B_\nn B_\nn^*)^{-1}-Z_\nn^\dag\}_n\sim_{\textup{GLT}} \kappa^*(\kappa\kappa^*)^{-1} = \kappa^\dag, \]
and the theorem is proved.
\end{proof}

\subsection{Convergence results for rectangular GLT sequences}
The next theorem proves the analog of {\bf GLT\,7} for rectangular GLT sequences.

\begin{theorem}\label{GLT7}
Let $\{A_\nn\}_n$ be a $d$-level $(r,s)$-block matrix-sequence and let $\kappa:[0,1]^d\times[-\pi,\pi]^d\to\mathbb C^{r\times s}$ be measurable. Suppose that
\begin{itemize}[leftmargin=*]
	\item $\{B_{\nn,m}\}_n\sim_{\textup{GLT}}\kappa_m$,
	\item $\{B_{\nn,m}\}_n\xrightarrow{\textup{a.c.s.}}\{A_\nn\}_n$,
	\item $\kappa_m\to\kappa$ in measure.
\end{itemize}
Then $\{A_\nn\}_n\sim_{\textup{GLT}}\kappa$.
\end{theorem}
\begin{proof}
Let $t\ge r\vee s$. We have
\begin{itemize}[leftmargin=*]
	\item $\{E_{r,s}^t(B_{\nn,m})\}_n\sim_{\textup{GLT}}E_{r,s}^t(\kappa_m)$ by Definition~\ref{rGLT},
	\item $\{E_{r,s}^t(B_{\nn,m})\}_n\xrightarrow{\textup{a.c.s.}}\{E_{r,s}^t(A_\nn)\}_n$ by Proposition~\ref{e.a.c.s.},
	\item $E_{r,s}^t(\kappa_m)\to E_{r,s}^t(\kappa)$ in measure (obviously).
\end{itemize}
We conclude by {\bf GLT\,7} that $\{E_{r,s}^t(A_\nn)\}_n\sim_{\textup{GLT}}E_{r,s}^t(\kappa)$, and so $\{A_\nn\}_n\sim_{\textup{GLT}}\kappa$ by Definition~\ref{rGLT}.
\end{proof}

The next theorem proves the analog of {\bf GLT\,8} for rectangular GLT sequences.

\begin{theorem}\label{GLT8}
Let $\{A_\nn\}_n\sim_{\textup{GLT}}\kappa$ and $\{B_{\nn,m}\}_n\sim_{\textup{GLT}}\kappa_m$. Then,
\[ \{B_{\nn,m}\}_n\xrightarrow{\textup{a.c.s.}}\{A_\nn\}_n\ \iff\ \kappa_m\to\kappa\mbox{ \,in measure}. \]
\end{theorem}
\begin{proof}
Let $t\ge r\vee s$, where $r\times s$ is the size of $\kappa$ and $\kappa_m$. By Definition~\ref{rGLT}, we have $\{E_{r,s}^t(A_\nn)\}_n\sim_{\textup{GLT}}E_{r,s}^t(\kappa)$ and $\{E_{r,s}^t(B_{\nn,m})\}_n\sim_{\textup{GLT}}E_{r,s}^t(\kappa_m)$. Thus, by Proposition~\ref{e.a.c.s.} and {\bf GLT\,8},
\begin{align*}
\{B_{\nn,m}\}_n\xrightarrow{\textup{a.c.s.}}\{A_\nn\}_n&\ \iff\ \{E_{r,s}^t(B_{\nn,m})\}_n\xrightarrow{\textup{a.c.s.}}\{E_{r,s}^t(A_\nn)\}_n\\\
&\ \iff\ E_{r,s}^t(\kappa_m)\to E_{r,s}^t(\kappa)\mbox{ \,in measure}\\
&\ \iff\ \kappa_m\to\kappa\mbox{ \,in measure}
\end{align*}
and the theorem is proved.
\end{proof}

The next theorem proves the analog of {\bf GLT\,9} for rectangular GLT sequences.

\begin{theorem}\label{GLT9}
Let $\{A_\nn\}_n\sim_{\textup{GLT}}\kappa$. Then, there exist functions $a_{i,m}$, $f_{i,m}$, $i=1,\ldots,N_m$, such that
\begin{itemize}[leftmargin=*]
	\item $a_{i,m}:[0,1]^d\to\mathbb C$ belongs to $C^\infty([0,1]^d)$ and $f_{i,m}$ is a trigonometric monomial in $\{{\rm e}^{{\rm i}\jj\cdot\btheta}E_{\alpha\beta}^{(r,s)}:\jj\in\mathbb Z^d,\ 1\le\alpha\le r,\ 1\le \beta\le s\}$ with $r\times s$ being the size of $\kappa$,
	\item $\kappa_m(\xx,\btheta)=\sum_{i=1}^{N_m}a_{i,m}(\xx)f_{i,m}(\btheta)\to\kappa(\xx,\btheta)$ a.e., \vspace{3pt}
	\item $\{B_{\nn,m}\}_n=\bigl\{\sum_{i=1}^{N_m}D_\nn(a_{i,m}I_r)T_\nn(f_{i,m})\bigr\}_n\xrightarrow{\textup{a.c.s.}}\{A_\nn\}_n$.
\end{itemize}
\end{theorem}
\begin{proof}
By Lemma~\ref{l2.4}, there exist functions $a_{i,m}$, $f_{i,m}$, $i=1,\ldots,N_m$, such that $a_{i,m}:[0,1]^d\to\mathbb C$ belongs to $C^\infty([0,1]^d)$, $f_{i,m}$ is a trigonometric monomial in $\{{\rm e}^{{\rm i}\jj\cdot\btheta}E_{\alpha\beta}^{(r,s)}:\jj\in\mathbb Z^d,\ 1\le\alpha\le r,\ 1\le \beta\le s\}$, and
\begin{equation*}
\kappa_m(\xx,\btheta)=\sum_{i=1}^{N_m}a_{i,m}(\xx)f_{i,m}(\btheta)\to\kappa(\xx,\btheta)\mbox{ \,a.e.}
\end{equation*}
Since $\{D_\nn(a_{i,m}I_r)\}_n\sim_{\textup{GLT}}a_{i,m}(\xx)I_r$ and $\{T_\nn(f_{i,m})\}_n\sim_{\textup{GLT}}f(\btheta)$ (see Section~\ref{GLT3}), Theorem~\ref{GLT4-parte1} yields
\[ \{B_{\nn,m}\}_n=\Biggl\{\sum_{i=1}^{N_m}D_\nn(a_{i,m}I_r)T_\nn(f_{i,m})\Biggr\}_n\sim_{\textup{GLT}}\kappa(\xx,\btheta). \]
We conclude that $\{B_{\nn,m}\}_n\xrightarrow{\textup{a.c.s.}}\{A_\nn\}_n$ by Theorem~\ref{GLT8}.
\end{proof}

\subsection{Relations between rectangular GLT sequences of different size}
In this section, 
we prove a stronger version of {\bf GLT\,6} for rectangular GLT sequences. It should be considered not only as the analog of {\bf GLT\,6} for rectangular GLT sequences but also as an addendum to the theory of square GLT sequences developed in \cite{bgd}.

\begin{theorem}\label{GLT6plus1}
Let $\{A_\nn=[a_{\ii\jj}^{(\nn)}]_{\ii,\jj=\bu}^\nn\}_n$ be a $d$-level $(r,s)$-block GLT sequence with symbol $\kappa$.
If we restrict each $r\times s$ block $a_{\ii\jj}^{(\nn)}$ to the same $\tilde r\times\tilde s$ submatrix $\tilde a_{\ii\jj}^{(\nn)}$, 
then we obtain a $d$-level $(\tilde r,\tilde s)$-block GLT sequence $\{\tilde A_\nn=[\tilde a_{\ii\jj}^{(\nn)}]_{\ii,\jj=\bu}^\nn\}_n$ whose symbol $\tilde\kappa$ is the corresponding $\tilde r\times\tilde s$ submatrix of $\kappa$.
\end{theorem}
\begin{proof}
Define
{\allowdisplaybreaks\begin{align*}
\chi_{r,\tilde r}&=\textup{diagonal }\{0,1\}\textup{-matrix of size }r\textup{ with }1\textup{ in the positions}\\*
&\hphantom{{}={}}\textup{corresponding to the chosen }\tilde r\textup{ rows},\\
\chi_{s,\tilde s}&=\textup{diagonal }\{0,1\}\textup{-matrix of size }s\textup{ with }1\textup{ in the positions}\\*
&\hphantom{{}={}}\textup{corresponding to the chosen }\tilde s\textup{ columns},\\
\alpha_{r,\tilde r}&=\textup{permutation matrix of size }r\textup{ that moves in order the chosen }\tilde r\textup{ rows}\\*
&\hphantom{{}={}}\textup{to the first }\tilde r\textup{ rows},\\
\beta_{s,\tilde s}&=\textup{permutation matrix of size }s\textup{ that moves in order the chosen }\tilde s\textup{ columns}\\*
&\hphantom{{}={}}\textup{to the first }\tilde s\textup{ columns}.
\end{align*}}%
By definition, we have
\begin{align*}
\alpha_{r,\tilde r}\chi_{r,\tilde r}a_{\ii\jj}^{(\nn)}\chi_{s,\tilde s}\beta_{s,\tilde s}&=\left[\begin{array}{cc}\tilde a_{\ii\jj}^{(\nn)} & O\\[3pt] O & O\end{array}\right]_{r\times s},\quad\ii,\jj=\bu,\ldots,\nn,\\
\alpha_{r,\tilde r}\chi_{r,\tilde r}\kappa\chi_{s,\tilde s}\beta_{s,\tilde s}&=\left[\begin{array}{cc}\tilde\kappa & O\\ O & O\end{array}\right]_{r\times s},
\end{align*}
where the subscript $r\times s$ indicates that the matrix size is $r\times s$.
Since we know that $\{D_\nn(\alpha_{r,\tilde r}\chi_{r,\tilde r})\}_n\sim_{\textup{GLT}}\alpha_{r,\tilde r}\chi_{r,\tilde r}$ and $\{D_\nn(\chi_{s,\tilde s}\beta_{s,\tilde s})\}_n\sim_{\textup{GLT}}\chi_{s,\tilde s}\beta_{s,\tilde s}$, Theorem~\ref{GLT4-parte1} yields
\begin{align}
\label{GLTtilde}\left\{\left[\left[\begin{array}{cc}\tilde a_{\ii\jj}^{(\nn)} & O\\[3pt] O & O\end{array}\right]_{r\times s}\right]_{\ii,\jj=\bu}^\nn\right\}_n&=\{D_\nn(\alpha_{r,\tilde r}\chi_{r,\tilde r})A_\nn D_\nn(\chi_{s,\tilde s}\beta_{s,\tilde s})\}_n\\
\notag&\sim_{\textup{GLT}}\alpha_{r,\tilde r}\chi_{r,\tilde r}\kappa\chi_{s,\tilde s}\beta_{s,\tilde s}=\left[\begin{array}{cc}\tilde\kappa & O\\ O & O\end{array}\right]_{r\times s}.
\end{align}
Let $t\ge r\vee s\ge\tilde r\vee\tilde s$. By \eqref{GLTtilde} and Definition~\ref{rGLT},
\begin{align*}
\{E_{\tilde r,\tilde s}^t(\tilde A_\nn)\}_n&=\{[E_{\tilde r,\tilde s}^t(\tilde a_{\ii\jj}^{(\nn)})]_{\ii,\jj=\bu}^\nn\}_n=\left\{\left[E_{r,s}^t\Biggl(\left[\begin{array}{cc}\tilde a_{\ii\jj}^{(\nn)} & O\\[3pt] O & O\end{array}\right]_{r\times s}\Biggr)\right]_{\ii,\jj=\bu}^\nn\right\}_n\\
&\sim_{\textup{GLT}}\left[\begin{array}{cc}\tilde\kappa & O\\ O & O\end{array}\right]_{t\times t}=E_{\tilde r,\tilde s}^t(\tilde\kappa).
\end{align*}
Thus, $\{\tilde A_\nn\}_n\sim_{\textup{GLT}}\tilde\kappa$ by Definition~\ref{rGLT}.
\end{proof}

\begin{theorem}\label{GLT6plus2}
Let $\{A_\nn=[a_{\ii\jj}^{(\nn)}]_{\ii,\jj=\bu}^\nn\}_n$ be a $d$-level $(r,s)$-block matrix-sequence and let
$\kappa:[0,1]^d\times[-\pi,\pi]^d\to\mathbb C^{r\times s}$ be measurable.
For $i=1,\ldots,r$ and $j=1,\ldots,s$, let $A_{\nn,ij}=[(a_{\ii\jj}^{(\nn)})_{ij}]_{\ii,\jj=\bu}^\nn$ be the submatrix of $A_\nn$ obtained by restricting each $r\times s$ block $a_{\ii\jj}^{(\nn)}$ to the $(i,j)$-entry $(a_{\ii\jj}^{(\nn)})_{ij}$. Then, 
\begin{align*}
&\{A_\nn\}_n\sim_{\textup{GLT}}\kappa\ \iff\ \{A_{\nn,ij}\}_n\sim_{\textup{GLT}}\kappa_{ij}\mbox{ \,for all \,}i=1,\ldots,r\mbox{ \,and \,}j=1,\ldots,s.
\end{align*}
\end{theorem}
\begin{proof}
($\implies$) This implication follows immediately from Theorem~\ref{GLT6plus1}.

($\impliedby$) Let $t\ge r\vee s$ and fix $(i,j)$ with $1\le i\le r$ and $1\le j\le s$. From the hypothesis $\{A_{\nn,ij}\}_n\sim_{\textup{GLT}}\kappa_{ij}$ and Definition~\ref{rGLT}, we have
\begin{align}\label{GLTij}
\{E_1^t(A_{\nn,ij})\}_n&=\{[E_1^t((a_{\ii\jj}^{(\nn)})_{ij})]_{\ii,\jj=\bu}^\nn\}_n=\left\{\left[\left[\begin{array}{cc}(a_{\ii\jj}^{(\nn)})_{ij} & O\\[3pt] O & O\end{array}\right]_{t\times t}\right]_{\ii,\jj=\bu}^\nn\right\}_n\\
\notag&\sim_{\textup{GLT}}E_1^t(\kappa_{ij})=\left[\begin{array}{cc}\kappa_{ij} & O\\[3pt] O & O\end{array}\right]_{t\times t}.
\end{align}
Let $\beta_i$ and $\beta_j$ be the permutation matrices that move the $(1,1)$ entry of a $t\times t$ matrix in position $(i,j)$, i.e.,
\begin{align*}
\beta_i&=\textup{permutation matrix of size }t\textup{ that swaps the rows }1\textup{ and }i,\\
\beta_j&=\textup{permutation matrix of size }t\textup{ that swaps the columns }1\textup{ and }j.
\end{align*}
By definition, we have
\begin{align*}
\beta_iE_1^t((a_{\ii\jj}^{(\nn)})_{ij})\beta_j&=\beta_i\left[\begin{array}{cc}(a_{\ii\jj}^{(\nn)})_{ij} & O\\[3pt] O & O\end{array}\right]_{t\times t}\beta_j
=(a_{\ii\jj}^{(\nn)})_{ij}E_{ij}^{(t)},\quad\ii,\jj=\bu,\ldots,\nn,\\
\beta_iE_1^t(\kappa_{ij})\beta_j&=\beta_i\left[\begin{array}{cc}\kappa_{ij} & O\\[3pt] O & O\end{array}\right]_{t\times t}\beta_j=\kappa_{ij}E_{ij}^{(t)}.
\end{align*}
Since $\{D_\nn(\beta_i)\}_n\sim_{\textup{GLT}}\beta_i$ and $\{D_\nn(\beta_j)\}_n\sim_{\textup{GLT}}\beta_j$, \eqref{GLTij} and Theorem~\ref{GLT4-parte1} yield
\begin{align*}
\{[(a_{\ii\jj}^{(\nn)})_{ij}E_{ij}^{(t)}]_{\ii,\jj=\bu}^\nn\}_n&=\{D_\nn(\beta_i)E_1^t(A_{\nn,ij})D_\nn(\beta_j)\}_n\sim_{\textup{GLT}}\beta_iE_1^t(\kappa_{ij})\beta_j=\kappa_{ij}E_{ij}^{(t)}.
\end{align*}
If we now sum over all $i=1,\ldots,r$ and $j=1,\ldots,s$, by the previous relation and Theorem~\ref{GLT4-parte1} we obtain
\[ \{E_{r,s}^t(A_\nn)\}_n=\left\{\left[\left[\begin{array}{cc}a_{\ii\jj}^{(\nn)} & O\\[3pt] O & O\end{array}\right]_{t\times t}\right]_{\ii,\jj=\bu}^\nn\right\}_n\sim_{\textup{GLT}}\left[\begin{array}{cc}\kappa & O\\ O & O\end{array}\right]_{t\times t}=E_{r,s}^t(\kappa). \]
We conclude that $\{A_\nn\}_n\sim_{\textup{GLT}}\kappa$ by Definition~\ref{rGLT}.
\end{proof}

\begin{theorem}\label{GLT6plus3}
For $i=1,\ldots,\varrho$ and $j=1,\ldots,\varsigma$, let $\{A_{\nn,ij}=[a_{\ii\jj,ij}^{(\nn)}]_{\ii,\jj=\bu}^\nn\}_n$ be a $d$-level $(r_i,s_j)$-block matrix-sequence and let $\kappa_{ij}:[0,1]^d\times[-\pi,\pi]^d\to\mathbb C^{r_i\times s_j}$ be measurable. Define the $(r,s)$-block matrix $A_\nn=[\,[a_{\ii\jj,ij}^{(\nn)}]_{i=1,\ldots,\varrho}^{j=1,\ldots,\varsigma}\,]_{\ii,\jj=\bu}^\nn$ and the $r\times s$ matrix-valued function $\kappa=[\kappa_{ij}]_{i=1,\ldots,\varrho}^{j=1,\ldots,\varsigma}$, where $r=\sum_{i=1}^\varrho r_i$ and $s=\sum_{j=1}^\varsigma s_j$.
Then,
\begin{equation}\label{glt6iff}
\{A_\nn\}_n\sim_{\textup{GLT}}\kappa\ \iff\ \{A_{\nn,ij}\}_n\sim_{\textup{GLT}}\kappa_{ij}\mbox{ \,for all \,}i=1,\ldots,\varrho\mbox{ \,and \,}j=1,\ldots,\varsigma.
\end{equation}
Moreover, if $B_\nn=[A_{\nn,ij}]_{i=1,\ldots,\varrho}^{j=1,\ldots,\varsigma}$ then
\begin{equation}\label{Prel}
\Bigl(P_{r,N(\nn)}\mathop{\textup{diag}}_{i=1,\ldots,\varrho}P_{r_i,N(\nn)}^T\Bigr)\,B_\nn\,\Bigl(P_{s,N(\nn)}\mathop{\textup{diag}}_{j=1,\ldots,\varsigma}P_{s_j,N(\nn)}^T\Bigr)^T=A_\nn,
\end{equation}
where $P_{k_1,k_2}$ is defined in \eqref{Pk1k2}.
\end{theorem}
\begin{proof}
We first prove the equivalence in \eqref{glt6iff}.

($\implies$) This implication follows immediately from Theorem~\ref{GLT6plus1}.

($\impliedby$) Let $A_{\nn,ij,\ell k}=[(a_{\ii\jj,ij}^{(\nn)})_{\ell k}]_{\ii,\jj=\bu}^\nn$.
By Theorem~\ref{GLT6plus1}, from the hypothesis $\{A_{\nn,ij}=[a_{\ii\jj,ij}^{(\nn)}]_{\ii,\jj=\bu}^\nn\}_n\sim_{\textup{GLT}}\kappa_{ij}$ we infer that
\[ \{A_{\nn,ij,\ell k}\}_n\sim_{\textup{GLT}}(\kappa_{ij})_{\ell k}. \]
Hence, the thesis $\{A_\nn\}_n\sim_{\textup{GLT}}\kappa$ follows from Theorem~\ref{GLT6plus2}.

We now prove \eqref{Prel}. We first note the following: if $[A_{ij}]_{i=1,\ldots,\varrho}^{j=1,\ldots,\varsigma}$ is a block matrix with $A_{ij}$ of size $r_i\times s_j$ and if we define $r=\sum_{i=1}^\varrho r_i$ and $s=\sum_{j=1}^\varsigma s_j$, then
\begin{align*}
&\sum_{p=1}^{r_i}\sum_{q=1}^{s_j}\ee_{p+r_1+\ldots+r_{i-1}}^{(r)}(\ee_p^{(r_i)})^TA_{ij}\ee_q^{(s_j)}(\ee_{q+s_1+\ldots+s_{j-1}}^{(s)})^T\\
&=\sum_{p=1}^{r_i}\sum_{q=1}^{s_j}\ee_{p+r_1+\ldots+r_{i-1}}^{(r)}(A_{ij})_{pq}(\ee_{q+s_1+\ldots+s_{j-1}}^{(s)})^T\\
&=\sum_{p=1}^{r_i}\sum_{q=1}^{s_j}(A_{ij})_{pq}E_{p+r_1+\ldots+r_{i-1},q+s_1+\ldots+s_{j-1}}^{(r,s)}\\
&=\begin{array}{c|c|c|c|c|c|}
\multicolumn{1}{c}{} & \multicolumn{1}{c}{s_1} & \multicolumn{1}{c}{\cdots} & \multicolumn{1}{c}{s_{j_{\vphantom{\sum_{\sum}}}}} & \multicolumn{1}{c}{\cdots} & \multicolumn{1}{c}{s_\varsigma}\\
\cline{2-6}
r_1\vphantom{\vdots_{\sum}} & O & O & O & O & O\\
\cline{2-6}
\vdots\vphantom{\vdots_{\sum}} & O & O & O & O & O\\
\cline{2-6}
r_i\vphantom{\vdots_{\sum}} & O & O & A_{ij_{\vphantom{\sum_{\sum}}}} & O & O\\
\cline{2-6}
\vdots\vphantom{\vdots_{\sum}} & O & O & O & O & O\\
\cline{2-6}
r_\varrho\vphantom{\vdots_{\sum}} & O & O & O & O & O\\
\cline{2-6}
\end{array}
\end{align*}
and
\begin{equation}\label{block.dec}	[A_{ij}]_{i=1,\ldots,\varrho}^{j=1,\ldots,\varsigma}=\sum_{i=1}^\varrho\sum_{j=1}^\varsigma\sum_{p=1}^{r_i}\sum_{q=1}^{s_j}\ee_{p+r_1+\ldots+r_{i-1}}^{(r)}(\ee_p^{(r_i)})^TA_{ij}\ee_q^{(s_j)}(\ee_{q+s_1+\ldots+s_{j-1}}^{(s)})^T.
\end{equation}
Let
\begin{align}\label{Bnij}
B_{\nn,ij}&=P_{r_i,N(\nn)}^TA_{\nn,ij}P_{s_j,N(\nn)}=P_{r_i,N(\nn)}^T[a_{\ii\jj,ij}^{(\nn)}]_{\ii,\jj=\bu}^\nn P_{s_j,N(\nn)}\\
&=P_{r_i,N(\nn)}^T\Biggl(\sum_{\ii,\jj=\bu}^\nn E_{\ii\jj}^{(\nn)}\otimes a_{\ii\jj,ij}^{(\nn)}\Biggr)P_{s_j,N(\nn)}\notag\\
&=\sum_{\ii,\jj=\bu}^\nn P_{r_i,N(\nn)}^T(E_{\ii\jj}^{(\nn)}\otimes a_{\ii\jj,ij}^{(\nn)})P_{s_j,N(\nn)}=\sum_{\ii,\jj=\bu}^\nn a_{\ii\jj,ij}^{(\nn)}\otimes E_{\ii\jj}^{(\nn)},\notag
\end{align}
where the last equality follows from \eqref{tensorP}. 
By \eqref{block.dec}, 
{\allowdisplaybreaks\begin{align*}
&\biggl(\mathop{\textup{diag}}_{i=1,\ldots,\varrho}P_{r_i,N(\nn)}^T\biggr)\,B_\nn\,\biggl(\mathop{\textup{diag}}_{j=1,\ldots,\varsigma}P_{s_j,N(\nn)}\biggr)\\*
&=\biggl(\mathop{\textup{diag}}_{i=1,\ldots,\varrho}P_{r_i,N(\nn)}^T\biggr)\,[A_{\nn,ij}]_{i=1,\ldots,\varrho}^{j=1,\ldots,\varsigma}\,\biggl(\mathop{\textup{diag}}_{j=1,\ldots,\varsigma}P_{s_j,N(\nn)}\biggr)\notag\\
&=[P_{r_i,N(\nn)}^TA_{\nn,ij}P_{s_j,N(\nn)}]_{i=1,\ldots,\varrho}^{j=1,\ldots,\varsigma}=[B_{\nn,ij}]_{i=1,\ldots,\varrho}^{j=1,\ldots,\varsigma}\notag\\
&=\sum_{i=1}^\varrho\sum_{j=1}^\varsigma\sum_{p'=1}^{N(\nn)r_i}\sum_{q'=1}^{N(\nn)s_j}\ee_{p'+N(\nn)r_1+\ldots+N(\nn)r_{i-1}}^{(N(\nn)r)}(\ee_{p'}^{(N(\nn)r_i)})^TB_{\nn,ij}\notag\\*
&\hphantom{{}=\sum_{i=1}^\varrho\sum_{j=1}^\varsigma\sum_{p'=1}^{N(\nn)r_i}\sum_{q'=1}^{N(\nn)s_j}}\cdot\ee_{q'}^{(N(\nn)s_j)}(\ee_{q'+N(\nn)s_1+\ldots+N(\nn)s_{j-1}}^{(N(\nn)s)})^T\notag\\
&=\sum_{i=1}^\varrho\sum_{j=1}^\varsigma\sum_{p=1}^{r_i}\sum_{u=1}^{N(\nn)}\sum_{q=1}^{s_j}\sum_{v=1}^{N(\nn)}\ee_{u+N(\nn)(p-1)+N(\nn)r_1+\ldots+N(\nn)r_{i-1}}^{(N(\nn)r)}(\ee_{u+N(\nn)(p-1)}^{(N(\nn)r_i)})^TB_{\nn,ij}\notag\\*
&\hphantom{{}=\sum_{i=1}^\varrho\sum_{j=1}^\varsigma\sum_{p=1}^{r_i}\sum_{u=1}^{N(\nn)}\sum_{q=1}^{s_j}\sum_{v=1}^{N(\nn)}}\cdot\ee_{v+N(\nn)(q-1)}^{(N(\nn)s_j)}(\ee_{v+N(\nn)(q-1)+N(\nn)s_1+\ldots+N(\nn)s_{j-1}}^{(N(\nn)s)})^T,
\end{align*}}%
where in the last equality we have used the changes of variable $p'=u+N(\nn)(p-1)$ and $q'=v+N(\nn)(q-1)$. Note that
\begin{align*}
\ee_{u+N(\nn)(p-1)+N(\nn)r_1+\ldots+N(\nn)r_{i-1}}^{(N(\nn)r)}&=\ee_{p+r_1+\ldots+r_{i-1}}^{(r)}\otimes\ee_u^{(N(\nn))},\\
\ee_{u+N(\nn)(p-1)}^{(N(\nn)r_i)}&=\ee_p^{(r_i)}\otimes\ee_u^{(N(\nn))},\\
\ee_{v+N(\nn)(q-1)+N(\nn)s_1+\ldots+N(\nn)s_{j-1}}^{(N(\nn)s)}&=\ee_{q+s_1+\ldots+s_{j-1}}^{(s)}\otimes\ee_v^{(N(\nn))},\\
\ee_{v+N(\nn)(q-1)}^{(N(\nn)s_j)}&=\ee_q^{(s_j)}\otimes\ee_v^{(N(\nn))},
\end{align*}
and
\begin{align*}
\sum_{u=1}^{N(\nn)}\ee_u^{(N(\nn))}(\ee_u^{(N(\nn))})^T=\sum_{v=1}^{N(\nn)}\ee_v^{(N(\nn))}(\ee_v^{(N(\nn))})^T=I_{N(\nn)}.
\end{align*}
Hence, by the properties \eqref{tensorT}--\eqref{tensorM} of tensor products, 
{\allowdisplaybreaks\begin{align*}
&\biggl(\mathop{\textup{diag}}_{i=1,\ldots,\varrho}P_{r_i,N(\nn)}^T\biggr)\,B_\nn\,\biggl(\mathop{\textup{diag}}_{j=1,\ldots,\varsigma}P_{s_j,N(\nn)}\biggr)\\*
&=\sum_{i=1}^\varrho\sum_{j=1}^\varsigma\sum_{p=1}^{r_i}\sum_{u=1}^{N(\nn)}\sum_{q=1}^{s_j}\sum_{v=1}^{N(\nn)}(\ee_{p+r_1+\ldots+r_{i-1}}^{(r)}\otimes\ee_u^{(N(\nn))})(\ee_p^{(r_i)}\otimes\ee_u^{(N(\nn))})^TB_{\nn,ij}\notag\\*
&\hphantom{{}=\sum_{i=1}^\varrho\sum_{j=1}^\varsigma\sum_{p=1}^{r_i}\sum_{u=1}^{N(\nn)}\sum_{q=1}^{s_j}\sum_{v=1}^{N(\nn)}}\cdot(\ee_q^{(s_j)}\otimes\ee_v^{(N(\nn))})(\ee_{q+s_1+\ldots+s_{j-1}}^{(s)}\otimes\ee_v^{(N(\nn))})^T\\
&=\sum_{i=1}^\varrho\sum_{j=1}^\varsigma\sum_{p=1}^{r_i}\sum_{u=1}^{N(\nn)}\sum_{q=1}^{s_j}\sum_{v=1}^{N(\nn)}(\ee_{p+r_1+\ldots+r_{i-1}}^{(r)}(\ee_p^{(r_i)})^T\otimes\ee_u^{(N(\nn))}(\ee_u^{(N(\nn))})^T)B_{\nn,ij}\notag\\*
&\hphantom{{}=\sum_{i=1}^\varrho\sum_{j=1}^\varsigma\sum_{p=1}^{r_i}\sum_{u=1}^{N(\nn)}\sum_{q=1}^{s_j}\sum_{v=1}^{N(\nn)}}\cdot(\ee_q^{(s_j)}(\ee_{q+s_1+\ldots+s_{j-1}}^{(s)})^T\otimes\ee_v^{(N(\nn))}(\ee_v^{(N(\nn))})^T)\\
&=\sum_{i=1}^\varrho\sum_{j=1}^\varsigma\sum_{p=1}^{r_i}\sum_{q=1}^{s_j}(\ee_{p+r_1+\ldots+r_{i-1}}^{(r)}(\ee_p^{(r_i)})^T\otimes I_{N(\nn)})B_{\nn,ij}\\*
&\hphantom{{}=\sum_{i=1}^\varrho\sum_{j=1}^\varsigma\sum_{p=1}^{r_i}\sum_{q=1}^{s_j}}\cdot(\ee_q^{(s_j)}(\ee_{q+s_1+\ldots+s_{j-1}}^{(s)})^T\otimes I_{N(\nn)}).
\end{align*}}%
Using \eqref{tensorP}, \eqref{block.dec}, and the expression \eqref{Bnij} for $B_{\nn,ij}$, we finally obtain
{\allowdisplaybreaks\begin{align*}
&P_{r,N(\nn)}\biggl(\mathop{\textup{diag}}_{i=1,\ldots,\varrho}P_{r_i,N(\nn)}^T\biggr)\,B_\nn\,\biggl(\mathop{\textup{diag}}_{j=1,\ldots,\varsigma}P_{s_j,N(\nn)}\biggr)P_{s,N(\nn)}^T\\*
&=P_{r,N(\nn)}\Biggl[\sum_{i=1}^\varrho\sum_{j=1}^\varsigma\sum_{p=1}^{r_i}\sum_{q=1}^{s_j}(\ee_{p+r_1+\ldots+r_{i-1}}^{(r)}(\ee_p^{(r_i)})^T\otimes I_{N(\nn)})\Biggl(\sum_{\ii,\jj=\bu}^\nn a_{\ii\jj,ij}^{(\nn)}\otimes E_{\ii\jj}^{(\nn)}\Biggr)\\*
&\hphantom{{}=P_{r,N(\nn)}\Biggl[\sum_{i=1}^\varrho\sum_{j=1}^\varsigma\sum_{p=1}^{r_i}\sum_{q=1}^{s_j}}\cdot(\ee_q^{(s_j)}(\ee_{q+s_1+\ldots+s_{j-1}}^{(s)})^T\otimes I_{N(\nn)})\Biggr]P_{s,N(\nn)}^T\\
&=P_{r,N(\nn)}\Biggl[\sum_{\ii,\jj=\bu}^\nn\sum_{i=1}^\varrho\sum_{j=1}^\varsigma\sum_{p=1}^{r_i}\sum_{q=1}^{s_j}\ee_{p+r_1+\ldots+r_{i-1}}^{(r)}(\ee_p^{(r_i)})^Ta_{\ii\jj,ij}^{(\nn)}\ee_q^{(s_j)}(\ee_{q+s_1+\ldots+s_{j-1}}^{(s)})^T\\*
&\hphantom{{}=P_{r,N(\nn)}\Biggl[\sum_{\ii,\jj=\bu}^\nn\sum_{i=1}^\varrho\sum_{j=1}^\varsigma\sum_{p=1}^{r_i}\sum_{q=1}^{s_j}}\otimes E_{\ii\jj}^{(\nn)})\Biggr]P_{s,N(\nn)}^T\\
&=P_{r,N(\nn)}\Biggl[\sum_{\ii,\jj=\bu}^\nn[a_{\ii\jj,ij}^{(\nn)}]_{i=1,\ldots,\varrho}^{j=1,\ldots,\varsigma}\otimes E_{\ii\jj}^{(\nn)}\Biggr]P_{s,N(\nn)}^T=\sum_{\ii,\jj=\bu}^\nn E_{\ii\jj}^{(\nn)}\otimes[a_{\ii\jj,ij}^{(\nn)}]_{i=1,\ldots,\varrho}^{j=1,\ldots,\varsigma}\\
&=[\,[a_{\ii\jj,ij}^{(\nn)}]_{i=1,\ldots,\varrho}^{j=1,\ldots,\varsigma}\,]_{\ii,\jj=\bu}^\nn=A_\nn,
\end{align*}}%
which proves the thesis \eqref{Prel}.
\end{proof}

\subsection{Existence of a rectangular GLT sequence for any measurable function}
The next theorem proves the analog of the second part of {\bf GLT\,0} for rectangular GLT sequences.

\begin{theorem}\label{GLT0-parte2}
Let $\{\nn=\nn(n)\}_n$ be a sequence of $d$-indices such that $\nn\to\infty$ as $n\to\infty$ and let $\kappa:[0,1]^d\times[-\pi,\pi]^d\to\mathbb C^{r\times s}$ be measurable. Then, there exists a $d$-level $(r,s)$-block GLT sequence $\{A_\nn\}_n\sim_{\textup{GLT}}\kappa$.
\end{theorem}
\begin{proof}
By {\bf GLT\,0}, for every $i=1,\ldots,r$ and $j=1,\ldots,s$ there exists $\{A_{\nn,ij}\}_n\sim_{\textup{GLT}}\kappa_{ij}$.
We define $B_\nn=[A_{\nn,ij}]_{i=1,\ldots,r}^{j=1,\ldots,s}$ and we conclude that $\{P_{r,N(\nn)}B_\nn P_{s,N(\nn)}^T\}_n\sim_{\textup{GLT}}\kappa$ by Theorem~\ref{GLT6plus3}. 
\end{proof}

\section{Summary of the theory of rectangular GLT sequences}\label{summ}
We summarize in this section the theory of rectangular GLT sequences developed in Section~\ref{sec:GLT}. By comparing this section with Section~\ref{sec:sGLT}, we see that all properties of square GLT sequences generalize to rectangular GLT sequences as long as they do not involve spectral symbols or Hermitian matrices. We remark that property \underline{\bf GLT\,6} below is a stronger version of {\bf GLT\,6} and should therefore be considered not only as a generalization of {\bf GLT\,6} to rectangular GLT sequences but also as an addendum to the theory of square GLT sequences developed in~\cite{bgd}.

A $d$-level $(r,s)$-block GLT sequence $\{A_\nn\}_n$ is a special $d$-level $(r,s)$-block matrix-sequence equipped with a measurable function $\kappa:[0,1]^d\times[-\pi,\pi]^d\to\mathbb C^{r\times s}$, the so-called symbol (or kernel). In the properties listed below, unless otherwise specified, 
the notation $\{A_\nn\}_n\sim_{\textup{GLT}}\kappa$ means that $\{A_\nn\}_n$ is a $d$-level $(r,s)$-block GLT sequence with symbol $\kappa$. 
\begin{enumerate}[leftmargin=35pt]
	\item[\underline{\textbf{GLT\,0}}\textbf{.}] If $\{A_\nn\}_n\sim_{\textup{GLT}}\kappa$ then $\{A_\nn\}_n\sim_{\textup{GLT}}\xi$ if and only if $\kappa=\xi$ a.e.
	
	If $\kappa:[0,1]^d\times[-\pi,\pi]^d\to\mathbb C^{r\times s}$ is measurable and $\{\nn=\nn(n)\}_n$ is a sequence of $d$-indices such that $\nn\to\infty$ as $n\to\infty$ then there exists $\{A_\nn\}_n\sim_{\textup{GLT}}\kappa$.
	\item[\underline{\textbf{GLT\,1}}\textbf{.}] If $\{A_\nn\}_n\sim_{\textup{GLT}}\kappa$ then $\{A_\nn\}_n\sim_\sigma\kappa$.
	\item[\underline{\textbf{GLT\,3}}\textbf{.}] For every sequence of $d$-indices $\{\nn=\nn(n)\}_n$ such that $\nn\to\infty$ as $n\to\infty$,
	\begin{itemize}[leftmargin=*]
		\item $\{T_\nn(f)\}_n\sim_{\textup{GLT}}\kappa(\xx,\btheta)=f(\btheta)$ if $f:[-\pi,\pi]^d\to\mathbb C^{r\times s}$ is in $L^1([-\pi,\pi]^d)$,
		\item $\{D_\nn(a)\}_n\sim_{\textup{GLT}}\kappa(\xx,\btheta)=a(\xx)$ if $a:[0,1]^d\to\mathbb C^{r\times s}$ is continuous a.e., 
		\item $\{Z_\nn\}_n\sim_{\textup{GLT}}\kappa(\xx,\btheta)=O_{r,s}$ if and only if $\{Z_\nn\}_n\sim_\sigma0$.
	\end{itemize}
	\item[\underline{\textbf{GLT\,4}}\textbf{.}] 
	Suppose that $\{A_\nn\}_n\sim_{\textup{GLT}}\kappa$ and $\{B_\nn\}_n\sim_{\textup{GLT}}\xi$, where in this case $\kappa$ and $\xi$ may have sizes different from $r\times s$ and different from each other. Then,
	\begin{itemize}[leftmargin=*]
		\item $\{A_\nn^*\}_n\sim_{\textup{GLT}}\kappa^*$,
		\item $\{\alpha A_\nn+\beta B_\nn\}_n\sim_{\textup{GLT}}\alpha\kappa+\beta\xi$ for all $\alpha,\beta\in\mathbb C$ if $\kappa$ and $\xi$ are summable,
		\item $\{A_\nn B_\nn\}_n\sim_{\textup{GLT}}\kappa\xi$ if $\kappa$ and $\xi$ are multipliable,
		\item $\{A_\nn^\dag\}_n\sim_{\textup{GLT}}\kappa^\dag$ if $\kappa$ has full rank a.e.
	\end{itemize}
	\item[\underline{\textbf{GLT\,6}}\textbf{.}] If $\{A_\nn=[a_{\ii\jj}^{(\nn)}]_{\ii,\jj=\bu}^\nn\}_n$ is a $d$-level $(r,s)$-block GLT sequence with symbol $\kappa$ and we restrict each $r\times s$ block $a_{\ii\jj}^{(\nn)}$ to the same $\tilde r\times\tilde s$ submatrix $\tilde a_{\ii\jj}^{(\nn)}$, then we obtain a $d$-level $(\tilde r,\tilde s)$-block GLT sequence $\{\tilde A_\nn=[\tilde a_{\ii\jj}^{(\nn)}]_{\ii,\jj=\bu}^\nn\}_n$ whose symbol $\tilde\kappa$ is the corresponding $\tilde r\times\tilde s$ submatrix of $\kappa$.
	
	If $\{A_{\nn,ij}=[a_{\ii\jj,ij}^{(\nn)}]_{\ii,\jj=\bu}^\nn\}_n$ is a $d$-level $(r_i,s_j)$-block GLT sequence with symbol $\kappa_{ij}$ for $i=1,\ldots,\varrho$ and $j=1,\ldots,\varsigma$, and if $A_\nn=[\,[a_{\ii\jj,ij}^{(\nn)}]_{i=1,\ldots,\varrho}^{j=1,\ldots,\varsigma}\,]_{\ii,\jj=\bu}^\nn$, then $\{A_\nn\}_n$ is a $d$-level $(r,s)$-block GLT sequence with symbol $\kappa=[\kappa_{ij}]_{i=1,\ldots,\varrho}^{j=1,\ldots,\varsigma}$, where $r=\sum_{i=1}^\varrho r_i$ and $s=\sum_{j=1}^\varsigma s_j$. Moreover, if $B_\nn=[A_{\nn,ij}]_{i=1,\ldots,\varrho}^{j=1,\ldots,\varsigma}$ 
	then
	\[ \Bigl(P_{r,N(\nn)}\mathop{\textup{diag}}_{i=1,\ldots,\varrho}P_{r_i,N(\nn)}^T\Bigr)\,B_\nn\,\Bigl(P_{s,N(\nn)}\mathop{\textup{diag}}_{j=1,\ldots,\varsigma}P_{s_j,N(\nn)}^T\Bigr)^T=A_\nn, \]
	where $P_{k_1,k_2}$ is the permutation matrix defined in \eqref{Pk1k2}.
	\item[\underline{\textbf{GLT\,7}}\textbf{.}] $\{A_\nn\}_n\sim_{\textup{GLT}}\kappa$ if and only if there exist 
	$\{B_{\nn,m}\}_n\sim_{\textup{GLT}}\kappa_m$ such that $\{B_{\nn,m}\}_n$\linebreak$\xrightarrow{\textup{a.c.s.}}\{A_\nn\}_n$ and $\kappa_m\to\kappa$ in measure.
	\item[\underline{\textbf{GLT\,8}}\textbf{.}] Suppose $\{A_\nn\}_n\sim_{\textup{GLT}}\kappa$ and $\{B_{\nn,m}\}_n\sim_{\textup{GLT}}\kappa_m$. Then, $\{B_{\nn,m}\}_n\xrightarrow{\textup{a.c.s.}}\{A_\nn\}_n$ if and only if $\kappa_m\to\kappa$ in measure.
	\item[\underline{\textbf{GLT\,9}}\textbf{.}] If $\{A_\nn\}_n\sim_{\textup{GLT}}\kappa$ then there exist functions $a_{i,m},f_{i,m},\ i=1,\ldots,N_m$, such~that
	\begin{itemize}[leftmargin=*]
		\item $a_{i,m}:[0,1]^d\to\mathbb C$ belongs to $C^\infty([0,1]^d)$ and $f_{i,m}$ is a trigonometric monomial in $\{{\rm e}^{{\rm i} \jj\cdot\btheta}E_{\alpha\beta}^{(r,s)}:\jj\in\mathbb Z^d,\ 1\le\alpha\le r,\ 1\le\beta\le s\}$, 
		\item $\kappa_m(\xx,\btheta)=\sum_{i=1}^{N_m}a_{i,m}(\xx)f_{i,m}(\btheta)\to\kappa(\xx,\btheta)$ a.e., \vspace{3pt}
		\item $\{B_{\nn,m}\}_n=\bigl\{\sum_{i=1}^{N_m}D_\nn(a_{i,m}I_r)T_\nn(f_{i,m})\bigr\}_n\xrightarrow{\textup{a.c.s.}}\{A_\nn\}_n$.
	\end{itemize}
\end{enumerate}

\section{Application to higher-order FE discretizations of systems of DEs}\label{appl}
We provide in this section an example of application of the theory of rectangular GLT sequences in the context of higher-order FE discretizations of systems of differential equations (DEs). The proposed example is an adapted version of the problems considered in \cite{ashkanCMAME,NM-iNS}, which in fact inspired the writing of this paper.

\subsection{Problem formulation} Consider the following system of DEs: 
\begin{equation*}
\left\{
\begin{aligned}
-(a(x)u'(x))'+v'(x)&=f(x), &\quad x\in(0,1),\\[3pt]
-u'(x)-\rho v(x)&=g(x), &\quad x\in(0,1),\\[3pt]
u(0)=0,\quad u(1)=0,& \\[3pt]
v(0)=0,\quad v(1)=0,&
\end{aligned}\right.
\end{equation*}
where $\rho$ is a constant and $a\in L^1([0,1])$. 
The corresponding weak form 
reads as follows: find $u,v\in H^1_0([0,1])$ such that, for all $w\in H^1_0([0,1])$, 
\begin{equation*}
\left\{\begin{aligned}
\textstyle{\int_0^1a(x)u'(x)w'(x){\rm d} x + \int_0^1v'(x)w(x){\rm d} x} &= \textstyle{\int_0^1f(x)w(x){\rm d} x,}\\[3pt]
\textstyle{-\int_0^1u'(x)w(x){\rm d} x - \rho\int_0^1v(x)w(x){\rm d} x} &= \textstyle{\int_0^1g(x)w(x){\rm d} x.}
\end{aligned}\right.
\end{equation*}

\subsection{Galerkin discretization}
We look for approximations $u_\UU,v_\VV$ of $u,v$ by choosing two finite dimensional vector spaces $\UU,\VV\subset H^1_0([0,1])$ and solving the following discrete problem: find $u_\UU\in\UU$ and $v_\VV\in\VV$ such that, for all $U\in\UU$ and $V\in\VV$,
\begin{equation*}
\left\{\begin{aligned}
\textstyle{\int_0^1a(x)u_\UU'(x)U'(x){\rm d}x + \int_0^1v_\VV'(x)U(x){\rm d}x} &= \textstyle{\int_0^1f(x)U(x){\rm d}x,}\\[3pt]
\textstyle{-\int_0^1u'(x)V(x){\rm d}x - \rho\int_0^1v(x)V(x){\rm d}x} &= \textstyle{\int_0^1g(x)V(x){\rm d} x.}
\end{aligned}\right.
\end{equation*}
Let $\{\varphi_1,\ldots,\varphi_N\}$ be a basis of $\UU$ and let $\{\psi_1,\ldots,\psi_M\}$ be a basis of $\VV$. Then, we can write $u_\UU=\sum_{j=1}^Nu_j\varphi_j$ and $v_\VV=\sum_{j=1}^Mv_j\psi_j$ for unique vectors $\uu=(u_1,\ldots,u_N)^T$ and $\vv=(v_1,\ldots,v_M)^T$. By linearity, the computation of $u_\UU,v_\VV$ (i.e., of $\uu,\vv$) reduces to solving the linear system
\[ A_{N,M}\begin{bmatrix}\uu\\\vv\end{bmatrix}=\begin{bmatrix}\mathbf f\\\mathbf g\end{bmatrix}, \]
where $\mathbf f=\bigl[\int_0^1f(x)\varphi_i(x){\rm d}x\bigr]_{i=1}^N$, $\mathbf g=\bigl[\int_0^1g(x)\psi_i(x){\rm d}x\bigr]_{i=1}^M$,
\begin{equation}\label{anm}
A_{N,M}=\begin{bmatrix}A_N(1,1) & A_{N,M}(1,2)\\
A_{N,M}(2,1) & A_M(2,2)\end{bmatrix}=\begin{bmatrix}A_N(1,1) & A_{N,M}(1,2)\\
(A_{N,M}(1,2))^T & A_M(2,2)\end{bmatrix},
\end{equation}
and
{\allowdisplaybreaks\begin{align}
\label{an11}A_N(1,1)&=\left[\int_0^1a(x)\varphi_j'(x)\varphi_i'(x){\rm d}x\right]_{i,j=1}^N,\\
\label{anm12}A_{N,M}(1,2)&=\left[\int_0^1\psi_j'(x)\varphi_i(x){\rm d}x\right]_{i=1,\ldots,N}^{j=1,\ldots,M},\\
\label{anm21}A_{N,M}(2,1)&=\left[-\int_0^1\varphi_j'(x)\psi_i(x){\rm d}x\right]_{i=1,\ldots,M}^{j=1,\ldots,N}\\*
\notag&=\left[\int_0^1\varphi_j(x)\psi_i'(x){\rm d}x\right]_{i=1,\ldots,M}^{j=1,\ldots,N}=(A_{N,M}(1,2))^T,\\
\label{am22}A_M(2,2)&=\left[-\rho\int_0^1\psi_j(x)\psi_i(x){\rm d}x\right]_{i,j=1}^M.
\end{align}}%
Assuming that $A_{N,M}(1,1)$ is invertible, the Schur complement of $A_{N,M}$ is the symmetric matrix given by
\begin{align}\label{snm}
S_{N,M}&=A_M(2,2)-A_{N,M}(2,1)(A_N(1,1))^{-1}A_{N,M}(1,2)\\
\notag&=A_M(2,2)-(A_{N,M}(1,2))^T(A_N(1,1))^{-1}A_{N,M}(1,2).
\end{align}

\begin{remark}
Suppose that $N=N_n$ and $M=M_n$ depend on a unique fineness parameter $n$.
If $\UU=\VV$ and $\{\varphi_1,\ldots,\varphi_N\}=\{\psi_1,\ldots,\psi_M\}$, then the sequence $\{A_{N,M}(i,j)\}_n$ is, up to minor transformations, a square GLT sequence for every $i,j=1,2$. In this case, the spectral distributions of $\{A_{N,M}\}_n$ and $\{S_{N,M}\}_n$ 
can be computed through the theory of square GLT sequences, without resorting to rectangular GLT sequences; see \cite[Section~6.4]{bg} and \cite[Section~10.6.2]{GLT-bookI}. For stability reasons, however, it is often convenient to choose two different spaces $\UU,\VV$. This happens, for instance, when $\UU,\VV$ have to be chosen so that the Ladyzhenskaya--Babu$\check{\mbox{s}}$ka--Brezzi (LBB) stability condition is met \cite{BF}, as in the Taylor--Hood FE discretizations \cite{ashkanCMAME}. If $\UU,\VV$ are FE spaces of different orders, then $\{A_{N,M}(i,j)\}_n$ is, up to minor transformations, a rectangular GLT sequence for $i\ne j$, and the computation of the spectral distributions of $\{A_{N,M}\}_n$ and $\{S_{N,M}\}_n$ requires the theory of rectangular GLT sequences (especially \underline{\bf GLT\,4} and \underline{\bf GLT\,6}, which allow us to ``connect'' GLT sequences with symbols of different size).
\end{remark}

\subsection{B-spline basis functions}
Following the higher-order FE approach, the basis functions $\varphi_1,\ldots,\varphi_N$ and $\psi_1,\ldots,\psi_M$ are chosen as piecewise polynomials of degree $\ge1$. More precisely, for $p,n\ge1$ and $0\le k\le p-1$, let $B_{1,[p,k]},\ldots,B_{n(p-k)+k+1,[p,k]}:\mathbb R\to\mathbb R$ be the B-splines of degree $p$ and smoothness $C^k$ defined on the knot sequence
\begin{align*}
&\{\tau_1,\ldots,\tau_{n(p-k)+p+k+2}\}\\
\notag&=\biggl\{\underbrace{0,\ldots,0}_{p+1},\ \underbrace{\frac1n,\ldots,\frac1n}_{p-k},\ \underbrace{\frac2n,\ldots,\frac2n}_{p-k},\ \ldots,\ \underbrace{\frac{n-1}n,\ldots,\frac{n-1}n}_{p-k},\ \underbrace{1,\ldots,1}_{p+1}\biggr\}.
\end{align*}
We collect here a few properties of $B_{1,[p,k]},\ldots,B_{n(p-k)+k+1,[p,k]}$ that we shall need later on. For the formal definition of B-splines, as well as for the proof of the properties listed below, see~\cite{CAN1}. For more on spline functions, see~\cite{deBoor,CAGDtools,Schumaker}.
\begin{itemize}[leftmargin=*]
	\item The support of the $i$th B-spline is given by
	\begin{equation}\label{eq:support}
	\textup{supp}(B_{i,[p,k]})=[\tau_i,\tau_{i+p+1}],\quad i=1,\ldots,n(p-k)+k+1.
	\end{equation}
	\item Except for the first and the last one, all the other B-splines vanish on the boundary of $[0,1]$, i.e.,
	\begin{equation*}
	B_{i,[p,k]}(0)=B_{i,[p,k]}(1)=0,\quad i=2,\ldots,n(p-k)+k. 
	\end{equation*}
	\item $\{B_{i,[p,k]}:i=1,\ldots,n(p-k)+k+1\}$ is a basis for the space of piecewise polynomial functions on $[0,1]$ of degree $p$ and smoothness~$C^k$, that is,
	\begin{equation*}
	\SS_{n,[p,k]}=\bigl\{s\in C^k([0,1]):s|_{\left[\frac in,\frac{i+1}n\right]}\in\mathbb P_p\,\mbox{ for }\,i=0,\ldots,n-1\bigr\},
	\end{equation*}
	where $\mathbb P_p$ is the space of polynomials of degree $\le p$. Moreover, $\{B_{i,[p,k]}:i=2,\ldots,$\linebreak$n(p-k)+k\}$ is a basis for the space 
	\begin{equation*}
	\SS_{n,[p,k]}^0=\{s\in\SS_{n,[p,k]}:s(0)=s(1)=0\}.
	\end{equation*}
	\item All the B-splines, except for the first $k+1$ and the last $k+1$, are uniformly shifted-scaled versions of $p-k$ fixed reference functions $\beta_{1,[p,k]},\ldots,\beta_{p-k,[p,k]}$, namely the first $p-k$ B-splines defined on the reference knot sequence
	\begin{equation*} 
	\underbrace{0,\ldots,0}_{p-k},\underbrace{1,\ldots,1}_{p-k},\ldots,\underbrace{\eta(p,k),\ldots,\eta(p,k)}_{p-k},\quad\eta(p,k)=\left\lceil\frac{p+1}{p-k}\right\rceil.
	\end{equation*}
	The precise formula we shall need later on is the following: setting 
	\begin{equation*}
	\nu(p,k)=\left\lceil\frac{k+1}{p-k}\right\rceil,
	\end{equation*}
	for the B-splines $B_{k+2,[p,k]},\ldots,B_{k+1+(n-\nu(p,k))(p-k),[p,k]}$ we have
	\begin{equation}\label{eq:uniform-shift-scale}
	\begin{aligned}
	&B_{k+1+(p-k)(r-1)+t,[p,k]}(x)=\beta_{t,[p,k]}(nx-r+1),\\
	&r=1,\ldots,n-\nu(p,k),\quad t=1,\ldots,p-k.
	\end{aligned}
	\end{equation}
	We point out that the supports of the reference B-splines $\beta_{t,[p,k]}$ satisfy
	\begin{equation}\label{refsupp}
	\textup{supp}(\beta_{1,[p,k]})\subseteq\textup{supp}(\beta_{2,[p,k]})\subseteq\ldots\subseteq\textup{supp}(\beta_{p-k,[p,k]})=[0,\eta(p,k)].
	\end{equation}%
	\begin{figure}
\centering
\includegraphics[width=\textwidth]{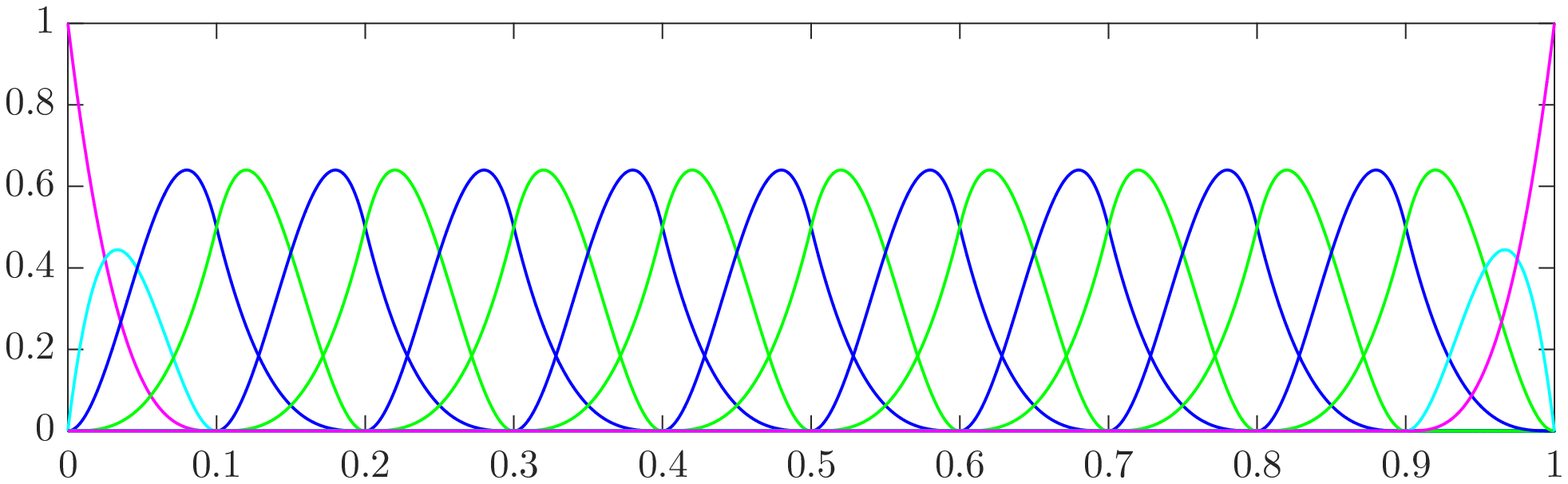}
\caption{B-splines $B_{1,[p,k]},\ldots,B_{n(p-k)+k+1,[p,k]}$ for $p=3$ and $k=1$, with $n=10$.}
\label{basis31}
\vspace{10pt}
\includegraphics[width=\textwidth]{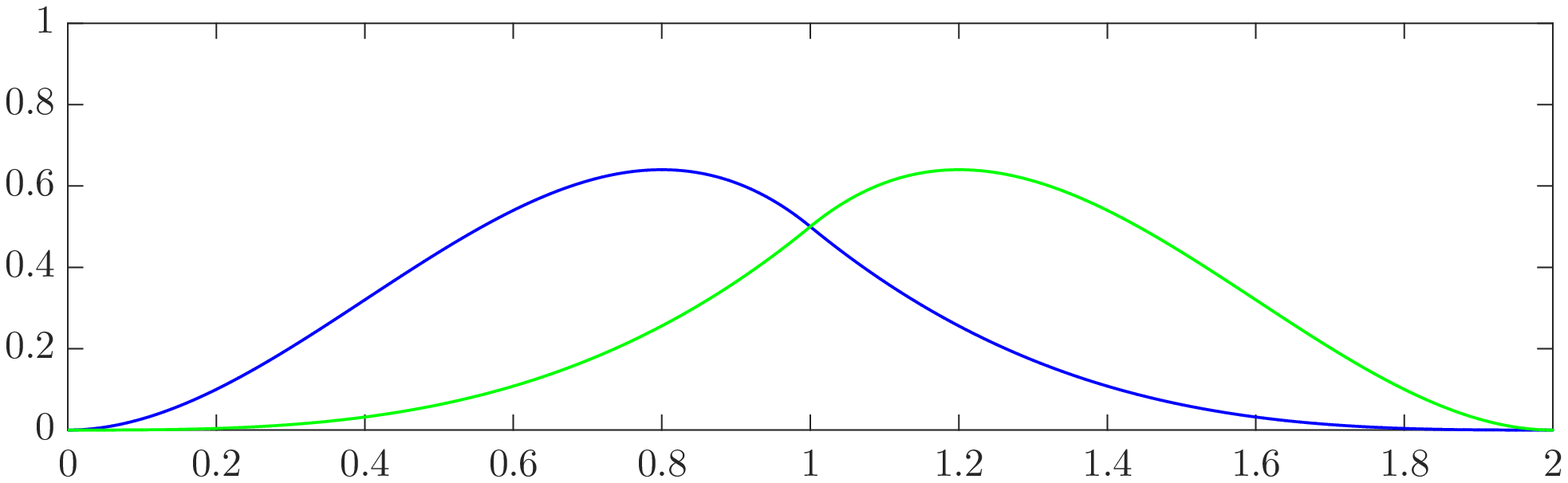}
\caption{Reference B-splines $\beta_{1,[p,k]},\beta_{2,[p,k]}$ for $p=3$ and $k=1$.}
\label{ref31}
\end{figure}%
	Figures~\ref{basis31}--\ref{ref31} display the graphs of the B-splines $B_{1,[p,k]},\ldots,B_{n(p-k)+k+1,[p,k]}$ for the degree $p=3$ and the smoothness $k=1$, and the graphs of the associated reference B-splines $\beta_{1,[p,k]},\beta_{2,[p,k]}$.
\end{itemize}
The basis functions $\varphi_1,\ldots,\varphi_N$ and $\psi_1,\ldots,\psi_M$ are defined as follows:
\begin{align}
\varphi_i&=B_{i+1,[p,k]},\quad i=1,\ldots,n(p-k)+k-1,\label{space-basisU}\\
\psi_i&=B_{i+1,[q,\ell]},\quad i=1,\ldots,n(q-\ell)+\ell-1.\label{space-basisV}
\end{align}
In particular, 
we have 
\begin{alignat*}{3}
\UU&=\textup{span}(\varphi_1,\ldots,\varphi_N)=\SS^0_{n,[p,k]}, &\quad N&=n(p-k)+k-1,\\
\VV&=\textup{span}(\psi_1,\ldots,\psi_M)=\SS^0_{n,[q,\ell]}, &\quad M&=n(q-\ell)+\ell-1.
\end{alignat*}

\subsection{GLT analysis of the higher-order FE discretization matrices}
The higher-order FE matrices \eqref{anm}--\eqref{snm} resulting from the choice of the basis functions as in \eqref{space-basisU}--\eqref{space-basisV} will be denoted by $A_n$, $A_n(1,1)$, $A_n(1,2)$, $A_n(2,1)$, $A_n(2,2)$, $S_n$, respectively. We therefore have $A_n(2,1)=(A_n(1,2))^T$ and
{\allowdisplaybreaks\begin{align*}
A_n&=\begin{bmatrix}A_n(1,1) & A_n(1,2)\\ (A_n(1,2))^T & A_n(2,2)\end{bmatrix},\\
A_n(1,1)&=\left[\int_0^1a(x)B'_{j+1,[p,k]}(x)B'_{i+1,[p,k]}(x){\rm d}x\right]_{i,j=1}^{n(p-k)+k-1},\\
A_n(1,2)&=\left[\int_0^1B'_{j+1,[q,\ell]}(x)B_{i+1,[p,k]}(x){\rm d}x\right]_{i=1,\ldots,n(p-k)+k-1}^{j=1,\ldots,n(q-\ell)+\ell-1},\\
A_n(2,2)&=\left[-\rho\int_0^1B_{j+1,[q,\ell]}(x)B_{i+1,[q,\ell]}(x){\rm d}x\right]_{i,j=1}^{n(q-\ell)+\ell-1},\\
S_n&=A_n(2,2)-(A_n(1,2))^T(A_n(1,1))^{-1}A_n(1,2).
\end{align*}}%
The main result of this section is Theorem~\ref{main}, which gives the singular value and spectral distributions of (properly normalized versions of) $\{A_n\}_n$ and $\{S_n\}_n$. If the sequences $\{n^{-1}A_n(1,1)\}_n$, $\{A_n(1,2)\}_n$, $\{nA_n(2,2)\}_n$
were exact (square or rectangular) GLT sequences, Theorem~\ref{main} would follow immediately from {\bf GLT\,1}, \underline{\bf GLT\,4} and \underline{\bf GLT\,6}. Unfortunately, the previous sequences 
are GLT sequences only up to minor transformations that, despite being minor, complicate the proof of Theorem~\ref{main} from a technical point of view. As we are going to see, the minor transformation we need to turn $\{n^{-1}A_n(1,1)\}_n$ into a GLT sequence is an expansion of each matrix $A_n(1,1)$ so as to reach the ``right'' size. 
The same applies to $\{A_n(1,2)\}_n$ and $\{nA_n(2,2)\}_n$. We remark that this expansion technique is quite common in the GLT context; see, e.g., \cite[Section~6]{bg} and \cite[Section~6]{bgd}.

\begin{notation}{\rm 
Fix a non-negative integer $m$ such that $m(p-k)\ge k$ and $m(q-\ell)\ge\ell$.\footnote{For example, take $m=k\vee\ell$.}
We denote by $\hat A_n(1,1)$ and $\hat A_n(2,2)$ the square block diagonal matrices obtained by expanding $A_n(1,1)$ and $A_n(2,2)$ as follows:
\begin{align*}
\hat A_n(1,1)&=\begin{bmatrix}I_{m(p-k)-k} & & \\ & A_n(1,1) & \\ & & \ \ \ 1\end{bmatrix}\in\mathbb R^{(n+m)(p-k)\times(n+m)(p-k)},\\
\hat A_n(2,2)&=\begin{bmatrix}I_{m(q-\ell)-\ell} & & \\ & A_n(2,2) & \\ & & \ \ \ 1\end{bmatrix}\in\mathbb R^{(n+m)(q-\ell)\times(n+m)(q-\ell)}.
\end{align*}
We denote by $\hat A_n(1,2)$ the rectangular block diagonal matrix obtained by expanding $A_n(1,2)$ as follows:
\[ \hat A_n(1,2)=\begin{bmatrix}O_{m(p-k)-k,m(q-\ell)-\ell} & & \\ & A_n(1,2) & \\ & & \ \ \ 0\end{bmatrix}\in\mathbb R^{(n+m)(p-k)\times(n+m)(q-\ell)}. \]
We denote by $\hat A_n$ and $\hat S_n$ the matrices obtained by expanding $A_n$ and $S_n$ as follows:
\begin{align*}
\hat A_n&=\begin{bmatrix}\hat A_n(1,1) & \hat A_n(1,2)\\ (\hat A_n(1,2))^T & \hat A_n(2,2)\end{bmatrix},\\
\hat S_n&=\hat A_n(2,2)-(\hat A_n(1,2))^T(\hat A_n(1,1))^{-1}\hat A_n(1,2); 
\end{align*}
see Figure~\ref{expansions}. 
\begin{figure}
\centering
\includegraphics[width=0.57\textwidth]{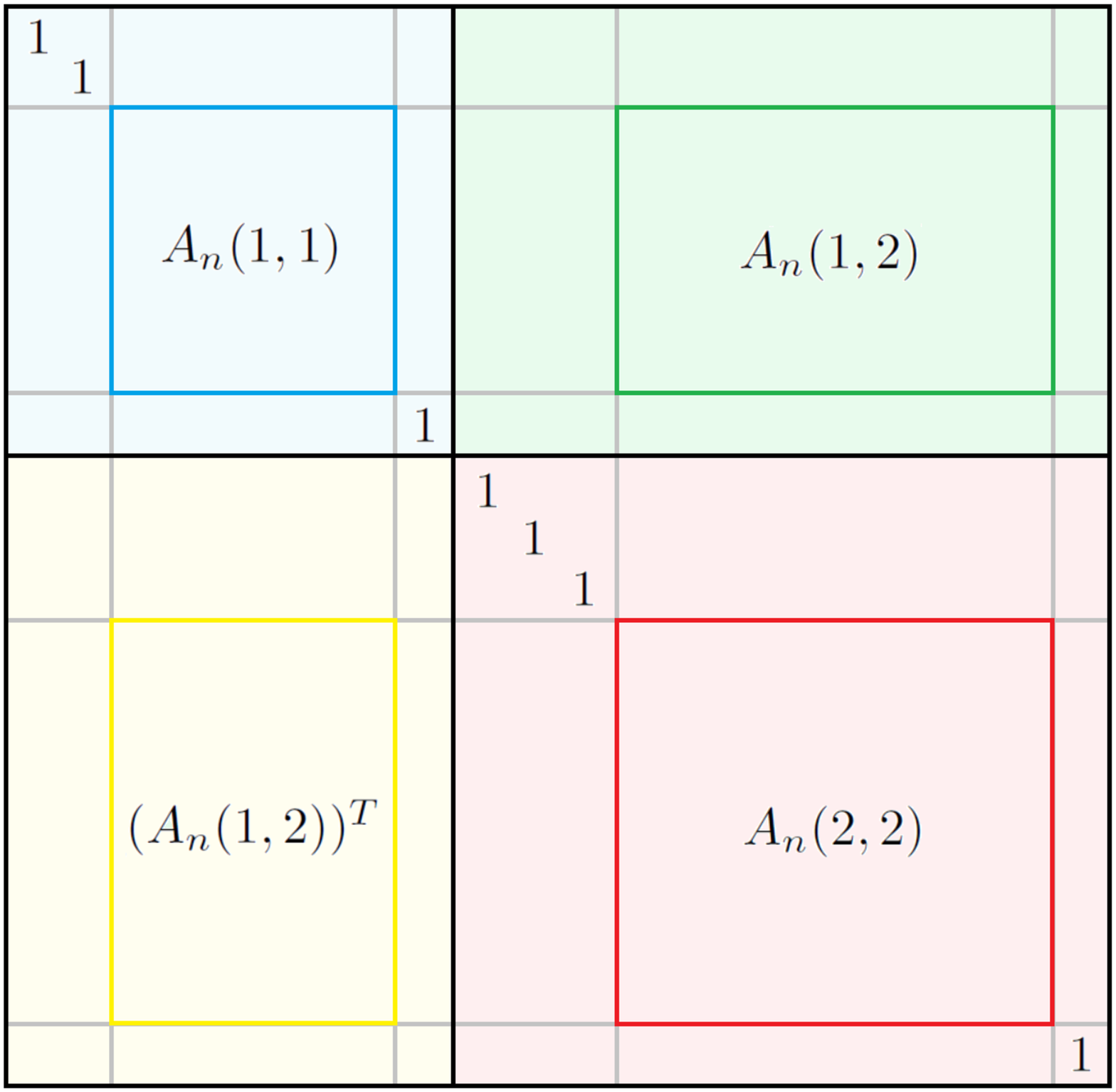}
\caption{Schematic representation of $\hat A_n$ in the case $m(p-k)-k=2$ and $m(q-\ell)-\ell=3$. The expanded matrices $\hat A_n(1,1)$, $\hat A_n(1,2)$, $(\hat A_n(1,2))^T$, $\hat A_n(2,2)$ are shaded respectively in azure, green, yellow, pink.}
\label{expansions}
\end{figure}
We define the blocks
\begin{alignat*}{3}
K_{[p,k]}^{[s]}&=\left[\int_{\mathbb R}\beta_{j,[p,k]}'(y)\beta_{i,[p,k]}'(y-s){\rm d}y\right]_{i,j=1}^{p-k}, &\quad s&\in\mathbb Z,\\
H_{[p,k;q,\ell]}^{[s]}&=\left[\int_{\mathbb R}\beta_{j,[q,\ell]}'(y)\beta_{i,[p,k]}(y-s){\rm d}y\right]_{i=1,\ldots,p-k}^{j=1,\ldots,q-\ell}, &\quad s&\in\mathbb Z,\\
M_{[q,\ell]}^{[s]}&=\left[\int_{\mathbb R}\beta_{j,[q,\ell]}(y)\beta_{i,[q,\ell]}(y-s){\rm d}y\right]_{i,j=1}^{q-\ell}, &\quad s&\in\mathbb Z,
\end{alignat*}
and the matrix-valued functions
\begin{alignat}{3}
\label{kappa.}\kappa_{[p,k]}&:[-\pi,\pi]\to\mathbb C^{(p-k)\times(p-k)}, &\quad\kappa_{[p,k]}(\theta)&=\sum_{s\in\mathbb Z}K_{[p,k]}^{[s]}{\rm e}^{{\rm i}s\theta},\\
\label{xi.}\xi_{[p,k;q,\ell]}&:[-\pi,\pi]\to\mathbb C^{(p-k)\times(q-\ell)}, &\quad\xi_{[p,k;q,\ell]}(\theta)&=\sum_{s\in\mathbb Z}H_{[p,k;q,\ell]}^{[s]}{\rm e}^{{\rm i}s\theta},\\
\label{mu.}\mu_{[q,\ell]}&:[-\pi,\pi]\to\mathbb C^{(q-\ell)\times(q-\ell)}, &\quad\mu_{[q,\ell]}(\theta)&=\sum_{s\in\mathbb Z}M_{[q,\ell]}^{[s]}{\rm e}^{{\rm i}s\theta}.
\end{alignat}
Due to the compact support of the reference B-splines 
(see \eqref{refsupp}), there are only a finite number of non-zero blocks $K_{[p,k]}^{[s]}$, $H_{[p,k;q,\ell]}^{[s]}$, $M_{[q,\ell]}^{[s]}$. Consequently, the series in \eqref{kappa.}--\eqref{mu.} are actually finite sums.}
\end{notation}

\begin{lemma}\label{real_main}
Let $a\in L^1([0,1])$, $\rho\in\mathbb R$, $p,q\ge1$, $0\le k\le p-1$ and $0\le\ell\le q-1$. Then,
\begin{align}
\label{GLT1,1}\{n^{-1}\hat A_n(1,1)\}_n&\sim_{\textup{GLT}}a(x)\kappa_{[p,k]}(\theta),\\
\label{GLT1,2}\{\hat A_n(1,2)\}_n&\sim_{\textup{GLT}}\xi_{[p,k;q,\ell]}(\theta),\\
\label{GLT2,1}\{(\hat A_n(1,2))^T\}_n&\sim_{\textup{GLT}}(\xi_{[p,k;q,\ell]}(\theta))^*,\\
\label{GLT2,2}\{n\hat A_n(2,2)\}_n&\sim_{\textup{GLT}}-\rho\mu_{[q,\ell]}(\theta).
\end{align}
\end{lemma}
\begin{proof}
We only have to prove \eqref{GLT1,2}. Indeed, $\{n^{-1}\hat A_n(1,1)\}_n$ and $\{n\hat A_n(2,2)\}_n$ are square GLT sequences and the proofs of \eqref{GLT1,1} and \eqref{GLT2,2} have already been given in \cite[Lemma~6.12]{bg}. Moreover, the GLT relation \eqref{GLT2,1} follows immediately from \eqref{GLT1,2} and \underline{\bf GLT\,4} (take into account that $(\hat A_n(1,2))^T=(\hat A_n(1,2))^*$ because $\hat A_n(1,2)$ is real). 

Let us then prove \eqref{GLT1,2}. By \eqref{eq:support}--\eqref{eq:uniform-shift-scale}, for every $r=1,\ldots,n-\nu(p,k),\ R=1,\ldots,$\linebreak$n-\nu(q,\ell)$ and every $t=1,\ldots,p-k,\ T=1,\ldots,q-\ell$, we have
\begin{align*}
&(\hat A_n(1,2))_{(p-k)(m+r-1)+t,(q-\ell)(m+R-1)+T}\\
&=(\hat A_n(1,2))_{[m(p-k)-k]+k+(p-k)(r-1)+t,[m(q-\ell)-\ell]+\ell+(q-\ell)(R-1)+T}\\
&=(A_n(1,2))_{k+(p-k)(r-1)+t,\ell+(q-\ell)(R-1)+T}\\
&=\int_0^1B_{\ell+1+(q-\ell)(R-1)+T,[q,\ell]}'(x)B_{k+1+(p-k)(r-1)+t,[p,k]}(x){\rm d}x\\
&=\int_{\mathbb R}B_{\ell+1+(q-\ell)(R-1)+T,[q,\ell]}'(x)B_{k+1+(p-k)(r-1)+t,[p,k]}(x){\rm d}x\\
&=\int_{\mathbb R}n\beta_{T,[q,\ell]}'(nx-R+1)\beta_{t,[p,k]}(nx-r+1){\rm d}x\\
&=\int_{\mathbb R}\beta_{T,[q,\ell]}'(y)\beta_{t,[p,k]}(y-r+R){\rm d}y=(H_{[p,k;q,\ell]}^{[r-R]})_{tT}\\
&=(T_{n+m}(\xi_{[p,k;q,\ell]}))_{(p-k)(m+r-1)+t,(q-\ell)(m+R-1)+T}.
\end{align*}
This means that the submatrix of $\hat A_n(1,2)$ corresponding to the row indices 
\[ i=m(p-k)+1,\ldots,(n+m-\nu(p,k))(p-k) \]
and the column indices 
\[ j=m(q-\ell)+1,\ldots,(n+m-\nu(q,\ell))(q-\ell) \]
coincides with the corresponding submatrix of $T_{n+m}(\xi_{[p,k;q,\ell]})$. Thus,
\[ \hat A_n(1,2) = T_{n+m}(\xi_{[p,k;q,\ell]})+R_n, \]
where $\textup{rank}(R_n)\le(m+\nu(p,k))(p-k)+(m+\nu(q,\ell))(q-\ell)=o(n)$. As a consequence, $\{R_n\}_n\sim_\sigma0$ by Definition~\ref{0d}. 
The thesis \eqref{GLT1,2} now follows from \underline{\bf GLT\,3}\,--\,\underline{\bf GLT\,4}.
\end{proof}

\begin{theorem}\label{main}
Let $a\in L^1([0,1])$, $\rho\in\mathbb R$, $p,q\ge1$, $0\le k\le p-1$ and $0\le\ell\le q-1$. Then,
\begin{equation}\label{An_sl}
\left\{\begin{bmatrix}
n^{-1}A_n(1,1) & A_n(1,2)\\
(A_n(1,2))^T & nA_n(2,2)
\end{bmatrix}\right\}_n\sim_{\sigma,\lambda}\begin{bmatrix}
a(x)\kappa_{[p,k]}(\theta) & \xi_{[p,k;q,\ell]}(\theta)\\
(\xi_{[p,k;q,\ell]}(\theta))^* & -\rho\mu_{[q,\ell]}(\theta)
\end{bmatrix}.
\end{equation}
Moreover, if the matrices $A_n(1,1)$ are invertible and $a\ne0$ a.e., then
\begin{equation}\label{Sn_sl}
\{nS_n\}_n\sim_{\sigma,\lambda}-\rho\mu_{[q,\ell]}(\theta)-\frac{(\xi_{[p,k;q,\ell]}(\theta))^*(\kappa_{[p,k]}(\theta))^{-1}\xi_{[p,k;q,\ell]}(\theta)}{a(x)}.
\end{equation}
\end{theorem}
\begin{proof}
We first prove \eqref{An_sl}. Consider the matrix obtained from the left-hand side of \eqref{An_sl} by replacing $A_n(1,1)$, $A_n(1,2)$, $A_n(2,2)$ with $\hat A_n(1,1)$, $\hat A_n(1,2)$, $\hat A_n(2,2)$. By Lemma~\ref{real_main} and \underline{\bf GLT\,6},
\[ \left\{\Pi_n\begin{bmatrix}
n^{-1}\hat A_n(1,1) & \hat A_n(1,2)\\
(\hat A_n(1,2))^T & n\hat A_n(2,2)
\end{bmatrix}\Pi_n^T\right\}_n\sim_{\textup{GLT}}\begin{bmatrix}
a(x)\kappa_{[p,k]}(\theta) & \xi_{[p,k;q,\ell]}(\theta)\\
(\xi_{[p,k;q,\ell]}(\theta))^* & -\rho\mu_{[q,\ell]}(\theta)
\end{bmatrix}, \]
where $\{\Pi_n\}_n$ is a sequence of permutation matrices. Hence, by {\bf GLT\,1},
\begin{align}\label{Anhat_sl}
\left\{\begin{bmatrix}
n^{-1}\hat A_n(1,1) & \hat A_n(1,2)\\
(\hat A_n(1,2))^T & n\hat A_n(2,2)
\end{bmatrix}\right\}_n\sim_{\sigma,\lambda}\begin{bmatrix}
a(x)\kappa_{[p,k]}(\theta) & \xi_{[p,k;q,\ell]}(\theta)\\
(\xi_{[p,k;q,\ell]}(\theta))^* & -\rho\mu_{[q,\ell]}(\theta)
\end{bmatrix}.
\end{align}
Looking at Figure~\ref{expansions}, we see that the singular values (resp., eigenvalues) of the matrix in the left-hand side of \eqref{Anhat_sl} are given by the singular values (resp., eigenvalues) of the matrix in the left-hand side of \eqref{An_sl} plus $m(p-k)-k+1$ singular values (resp., eigenvalues) that are equal to $n^{-1}$ plus $m(q-\ell)-\ell+1$ singular values (resp., eigenvalues) that are equal to $n$. Since $m(p-k)-k+m(q-\ell)-\ell+2$ is $o(n)$, \eqref{An_sl} follows from \eqref{Anhat_sl} and Definition~\ref{def-distribution}.

We now prove \eqref{Sn_sl}. The proof is completely analogous to the proof of \eqref{An_sl}. Consider the matrix 
\begin{align*}
n\hat S_n&=n(\hat A_n(2,2)-(\hat A_n(1,2))^T(\hat A_n(1,1))^{-1}\hat A_n(1,2))\\
&=n\hat A_n(2,2)-(\hat A_n(1,2))^T(n^{-1}\hat A_n(1,1))^{-1}\hat A_n(1,2).
\end{align*}
By Lemma~\ref{real_main} and \underline{\bf GLT\,4},
\[ \{n\hat S_n\}_n\sim_{\textup{GLT}}-\rho\mu_{[q,\ell]}(\theta)-\frac{(\xi_{[p,k;q,\ell]}(\theta))^*(\kappa_{[p,k]}(\theta))^{-1}\xi_{[p,k;q,\ell]}(\theta)}{a(x)}. \]
Hence, by {\bf GLT\,1},
\begin{equation}\label{Snhat_sl}
\{n\hat S_n\}_n\sim_{\sigma,\lambda}-\rho\mu_{[q,\ell]}(\theta)-\frac{(\xi_{[p,k;q,\ell]}(\theta))^*(\kappa_{[p,k]}(\theta))^{-1}\xi_{[p,k;q,\ell]}(\theta)}{a(x)}.
\end{equation}
Looking at Figure~\ref{expansions}, we see that the singular values (resp., eigenvalues) of $n\hat S_n$ are given by the singular values (resp., eigenvalues) of $nS_n$ plus $m(q-\ell)-\ell+1$ singular values (resp., eigenvalues) that are equal to $n$. Since $m(q-\ell)-\ell+1$ is $o(n)$, \eqref{Sn_sl} follows from \eqref{Snhat_sl} and Definition~\ref{def-distribution}.
\end{proof}


\begin{thebibliography}{99}

\bibitem{MJM}
{\sc A. Adriani, D. Bianchi, and S. Serra-Capizzano},
{\em Asymptotic spectra of large (grid) graphs with a uniform local structure (part~I): theory},
Milan J. Math., 88 (2020), pp. 409--454.

\bibitem{GiovanniLAA2017}
{\sc G. Barbarino},
{\em Equivalence between GLT sequences and measurable functions},
Linear Algebra Appl., 529 (2017), pp. 397--412.

\bibitem{GiovanniNF}
\sameauthor, 
{\em Normal form for GLT sequences},
arXiv:1805.08708, 2018.

\bibitem{barbarinoREDUCED}
\sameauthor, 
{\em A systematic approach to reduced GLT},
BIT Numer. Math. (2021) https://doi.org/10.1007/s10543-021-00896-7.

\bibitem{ELA}
{\sc G. Barbarino and C. Garoni},
{\em From convergence in measure to convergence of matrix-sequences through concave functions and singular values},
Electron. J. Linear Algebra, 32 (2017), pp. 500--513.

\bibitem{bg}
{\sc G. Barbarino, C. Garoni, and S. Serra-Capizzano},
{\em Block generalized locally Toeplitz sequences: theory and applications in the unidimensional case},
Electron. Trans. Numer. Anal., 53 (2020), pp. 28--112.

\bibitem{bgd}
\sameauthor, 
{\em Block generalized locally Toeplitz sequences: theory and applications in the multidimensional case},
Electron. Trans. Numer. Anal., 53 (2020), pp. 113--216.

\bibitem{barbarinoNH}
{\sc G. Barbarino and S. Serra-Capizzano},
{\em Non-Hermitian perturbations of Hermitian matrix-sequences and applications to the spectral analysis of the numerical approximation of partial differential equations},
Numer. Linear Algebra Appl., 27 (2020), e2286.

\bibitem{BGL}
{\sc M. Benzi, G. H. Golub, and J. Liesen},
{\em Numerical solution of saddle point problems},
Acta Numerica, 14 (2005), pp. 1--137.

\bibitem{bianchi}
{\sc D. Bianchi},
{\em Analysis of the spectral symbol associated to discretization schemes of linear self-adjoint differential operators},
Calcolo, 58 (2021), 38.

\bibitem{bianchi0}
{\sc D. Bianchi and S. Serra-Capizzano},
{\em Spectral analysis of finite-dimensional approximations of 1d waves in non-uniform grids},
Calcolo, 55 (2018), 47.

\bibitem{Bini}
{\sc D. Bini, M. Capovani, and O. Menchi},
{\em Metodi Numerici per l'Algebra Lineare},
Zanichelli, Bologna, 1988.

\bibitem{SbMath}
{\sc A. B\"ottcher, C. Garoni, and S. Serra-Capizzano},
{\em Exploration of Toeplitz-like matrices with unbounded symbols is not a purely academic journey},
Sb. Math., 208 (2017), pp. 1602--1627.

\bibitem{BF}
{\sc F. Brezzi and M. Fortin},
{\em Mixed and hybrid finite element methods},
Springer, New York, 1991.

\bibitem{deBoor}
{\sc C. De Boor},
{\em A Practical Guide to Splines},
revised ed., Springer, New York, 2001.

\bibitem{DMS}
{\sc M. Donatelli, M. Mazza, and S. Serra-Capizzano},
{\em Spectral analysis and structure preserving preconditioners for fractional diffusion equations},
J. Comput. Phys., 307 (2016), 262--279.

\bibitem{DMS2}
\sameauthor, 
{\em Spectral analysis and multigrid methods for finite volume approximations of space-fractional diffusion equations},
SIAM J. Sci. Comput., 40 (2018), A4007--A4039.

\bibitem{ashkanCMAME}
{\sc A. Dorostkar, M. Neytcheva, and S. Serra-Capizzano},
{\em Spectral analysis of coupled PDEs and of their Schur complements via generalized locally Toeplitz sequences in 2D},
Comput. Methods Appl. Mech. Engrg., 309 (2016), pp. 74--105.

\bibitem{DFF}
{\sc M. Dumbser, F. Fambri, I. Furci, M. Mazza, S. Serra-Capizzano, and M. Tavelli},
{\em Staggered discontinuous Galerkin methods for the incompressible Navier--Stokes equations: spectral analysis and computational results},
Numer. Linear Algebra Appl., 25 (2018), e2151.

\bibitem{axioms}
{\sc C. Garoni, M. Mazza, and S. Serra-Capizzano},
{\em Block generalized locally Toeplitz sequences: from the theory to the applications},
Axioms, 7 (2018), 49.

\bibitem{GLT-bookI}
{\sc C. Garoni and S. Serra-Capizzano},
{\em Generalized Locally Toeplitz Sequences: Theory and Applications. Vol. I},
Springer, Cham, 2017.

\bibitem{GLT-bookII}
\sameauthor, 
{\em Generalized Locally Toeplitz Sequences: Theory and Applications. Vol. II},
Springer, Cham, 2018.

\bibitem{cime}
\sameauthor, 
{\em Generalized locally Toeplitz sequences: a spectral analysis tool for discretized differential equations},
Lect. Notes Math., 2219 (2018), pp. 161--236.

\bibitem{aip}
\sameauthor, 
{\em Multilevel generalized locally Toeplitz sequences: an overview and an example of application},
AIP Conf. Proc., 2116 (2019), 020003.

\bibitem{Tom-paper}
{\sc C. Garoni, H. Speleers, S.-E. Ekstr\"om, A. Reali, S. Serra-Capizzano, and T. J. R. Hughes},
{\em Symbol-based analysis of finite element and isogeometric B-spline discretizations of eigevalue problems: exposition and review},
Arch. Comput. Methods Engrg., 26 (2019), pp. 1639--1690.

\bibitem{CAN1}
{\sc T. Lyche, C. Manni, and H. Speleers},
{\em Foundations of spline theory: B-splines, spline approximation, and hierarchical refinement},
Lect. Notes Math., 2219 (2018), pp. 1--76.

\bibitem{CAGDtools}
{\sc C. Manni and H. Speleers},
{\em Standard and non-standard CAGD tools in isogeometric analysis: a tutorial},
Lect. Notes Math., 2161 (2016), pp. 1--69.

\bibitem{NM-iNS}
{\sc M. Mazza, M. Semplice, S. Serra-Capizzano, and E. Travaglia},
{\em A matrix-theoretic spectral analysis of incompressible Navier--Stokes staggered DG approximations and a related spectrally based preconditioning approach},
Numer. Math., 149 (2021), pp. 933--971.

\bibitem{Schumaker}
{\sc L. L. Schumaker},
{\em Spline Functions: Basic Theory},
3rd ed., Cambridge University Press, Cambridge, 2007.

\bibitem{Ty96}
{\sc E. E. Tyrtyshnikov},
{\em A unifying approach to some old and new theorems on distribution and clustering},
Linear Algebra Appl., 232 (1996), pp. 1--43.

\end{thebibliography}
\end{document}